\documentclass[12pt, reqno]{amsart}

\usepackage{amsmath, amssymb, amscd, amsthm, comment, amsxtra}
\usepackage{graphicx}
\usepackage{subfigure}
\usepackage{enumerate}
\usepackage[pdftex, linktocpage=true]{hyperref}
\usepackage{euscript}
\usepackage[format=plain,labelfont=bf,up,width=.99\textwidth]{caption}
\usepackage[toc,page]{appendix}
\usepackage{xcolor}

\theoremstyle{plain}
\newtheorem{thm}{Theorem}[section]
\newtheorem*{mainthm}{Main Theorem}
\newtheorem{lem}[thm]{Lemma}
\newtheorem{prop}[thm]{Proposition}
\newtheorem{cor}[thm]{Corollary}

\theoremstyle{definition}
\newtheorem{defn}[thm]{Definition}

\newtheorem{rem}[thm]{Remark}

\newcommand{\Id}{\operatorname{Id}}
\newcommand{\Jac}{\operatorname{Jac}}
\newcommand{\loc}{\operatorname{loc}}
\newcommand{\rot}{\operatorname{rot}}

\numberwithin{equation}{section}
\newcommand{\thmref}[1]{Theorem~\ref{#1}}
\newcommand{\propref}[1]{Proposition~\ref{#1}}
\newcommand{\lemref}[1]{Lemma~\ref{#1}}
\newcommand{\corref}[1]{Corollary~\ref{#1}}
\newcommand{\figref}[1]{Figure~\ref{#1}}
\newcommand{\secref}[1]{Section~\ref{#1}}

\newcommand{\remref}[1]{Remark~\ref{#1}}

\newcommand{\bbN}{{\mathbb N}}
\newcommand{\bbZ}{{\mathbb Z}}
\newcommand{\bbQ}{{\mathbb Q}}
\newcommand{\bbR}{{\mathbb R}}
\newcommand{\bbC}{{\mathbb C}}
\newcommand{\bbD}{{\mathbb D}}

\newcommand{\cR}{{\mathcal R}}

\newcommand{\cD}{{\mathcal D}}
\newcommand{\cC}{{\mathcal C}}
\newcommand{\cM}{{\mathcal M}}
\newcommand{\cU}{{\mathcal U}}
\newcommand{\cG}{{\mathcal G}}

\newcommand{\cS}{{\mathcal S}}

\newcommand{\cT}{{\mathcal T}}

\newcommand{\bfR}{{\mathbf R}}

\newcommand{\bfx}{{\mathbf x}}

\newcommand{\bfq}{{\mathbf q}}
\newcommand{\bfs}{{\mathbf s}}
\newcommand{\bfv}{{\mathbf v}}

\title[Structural Instability in the Semi-Siegel H\'enon Family]{Structural Instability of Semi-Siegel H\'enon maps}
\author{Michael Yampolsky and Jonguk Yang}

\thanks{M. Y. was partially supported by NSERC Discovery grant. J. Y. carried out part of this work as a visiting member of the Fields Institute for Mathematical Sciences.}

\begin{document}

\begin{abstract}
We show that the dynamics of sufficiently dissipative semi-Siegel complex H\'enon maps with golden-mean rotation number is not $J$-stable in a very strong sense. By the work of Dujardin and Lyubich, this implies that the Newhouse phenomenon occurs for a dense $G_\delta$ set of parameters in this family. Another consequence is that the Julia sets of such maps are disconnected for a dense set of parameters.
\end{abstract}

\maketitle

\section*{Foreword}

Complex quadratic H\'enon maps have emerged as the key object of study in multi-dimensional Complex Dynamics. They occupy a similar place  to that of quadratic polynomials in one-dimensional dynamics. Indeed, highly dissipative H\'enon maps are naturally seen as small perturbations of quadratics -- and the underlying one-dimensional dynamics plays a major role in their study. As an example, Fornaess and Sibony \cite{FoSi}, and Hubbard and Oberste-Vorth \cite{HuOV} studied highly dissipative H\'enon maps with an attracting fixed point, and showed that the dynamics of such maps on the Julia set naturally projects on the dynamics of hyperbolic quadratic polynomials in the main component of the Mandelbrot set. Thus, the same topological model applies to the dynamics of all such two-dimensional maps. A parallel result for dissipative H\'enon maps with a semi-parabolic fixed point (that is, having an eigenvalue of the form $e^{2\pi i r},\; r\in\bbQ$) was later proven by Radu and Tanase \cite{RaTa}.

In the present paper we consider another type of examples with a semi-neutral fixed point, the semi-Siegel dissipative H\'enon maps. In the case when the neutral eigenvalue is equal to  $e^{2\pi i r}$, where $r=(\sqrt{5}+1)/2$ is the golden mean, Gaidashev, Radu, and the first author \cite{GaRaYam} recently showed that  the Siegel disk of the H\'enon map is bounded by a topological circle. This could be seen as a natural step towards describing the topological model of the dynamics of the two-dimensional map -- and leads to asking whether it is again fibered over the underlying one-dimensional dynamics.

As we will see, this is not the case. In fact, semi-Siegel H\'enon maps with golden-mean rotation numbers are structurally unstable in a very strong sense. The precise formulation of our result follows below. Combined with the work of Dujardin and Lyubich, our result has two interesting, previously unknown corollaries. First, the Newhouse phenomenon occurs for a dense $G_\delta$ set of parameters in the semi-Siegel H\'enon family. Second, for a dense set of parameters values, the golden-mean semi-Siegel H\'enon Julia set is disconnected.

\subsection*{Acknowledgement}

We would like to thank E. Bedford, M. Martens, and R. Radu for helpful discussions. We are grateful to R. Dujardin for pointing out the connections to his work with M. Lyubich, which enabled us to strengthen the results in an earlier version of this work.

\section{Introduction}
%
Consider the {\it complex quadratic H\'enon family}:
$$
F_{c,a}(x,y):=(x^2+c - ay,x) \hspace{5mm} \text{for } c \in \mathbb{C} \text{ and } a \in \mathbb{C} \setminus \{0\}.
$$
For $a\neq 0$, the map $F_{c,a}$ is an automorphism of $\bbC^2$ with a constant Jacobian:
$$
\Jac F_{c,a} \equiv a.
$$
Furthermore, we have
$$
F_{c,a}(x,y) \to (f_c(x),x)
\hspace{5mm} \text{as} \hspace{5mm}
a \to 0,
$$
where
$$
f_c(x) := x^2 +c \hspace{5mm} \text{for } c \in \bbC
$$
is the standard one-dimensional {\it quadratic family}. In this paper, we will always assume that $F_{c,a}$ is a {\it sufficiently dissipative} map (i.e. $|a| < \bar\epsilon \ll 1$). 

As is common, will refer to the $x$ coordinate of a point in $(x,y)\in \bbC^2$ as ``horizontal'', and the $y$ coordinate as ``vertical''. Figure~\ref{fig:henon} gives an intuition on how vertical and horizontal planes are mapped by a H\'enon map.
\begin{figure}[h]
\centering
\includegraphics[scale=0.3]{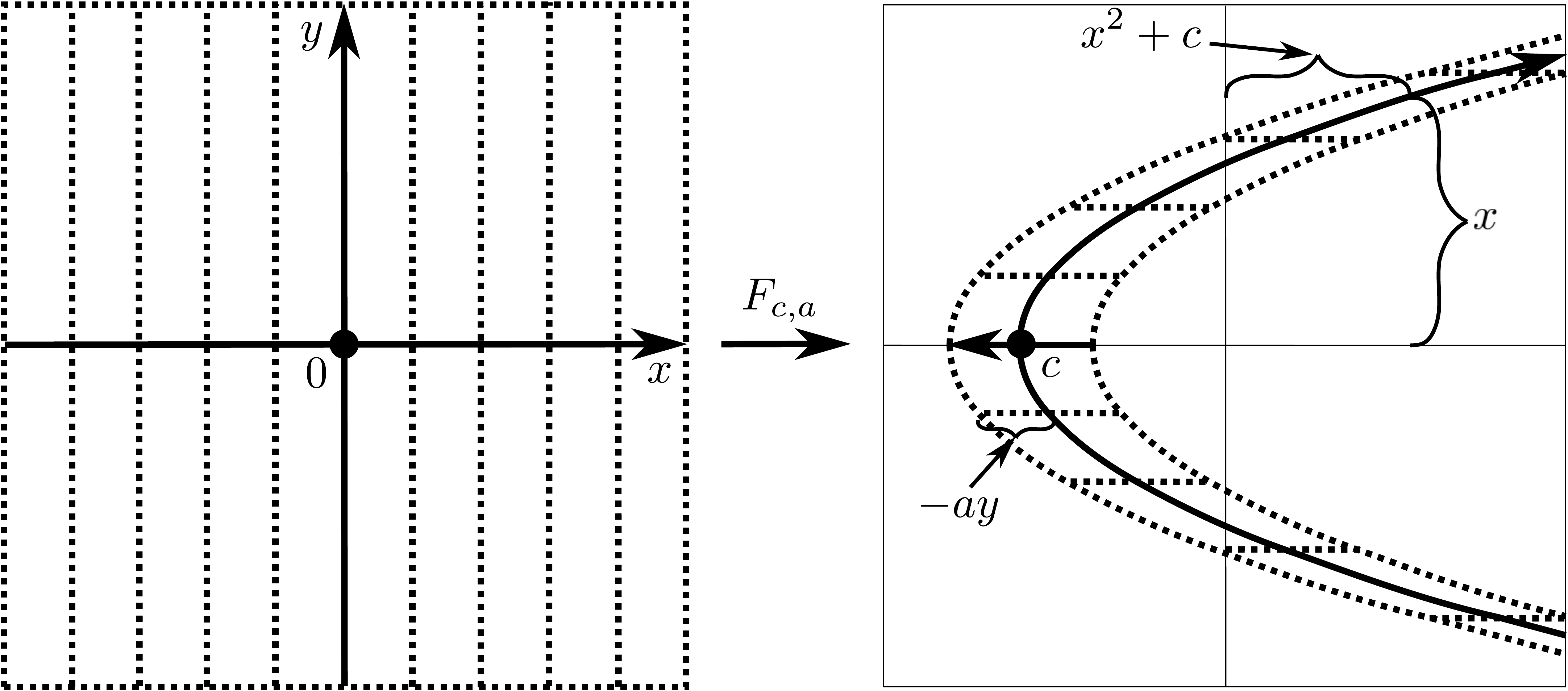}
\caption{A H\'enon map $F_{c,a}$. Note that the $x$-coordinate plane is mapped to the graph of $f_c$, and that the vertical planes are scaled uniformly by $-a$, and then mapped to horizontal planes.}
\label{fig:henon}
\end{figure}

In one-dimensional dynamics, the filled Julia set is defined as the set of points with bounded orbits. Since $F_{c,a}$ is an automorphism,  the analogous definitions are made both for forward and for backward orbits:
$$
K^\pm(F_{c,a}) := \{\bfx \in \bbC^2 \, | \, \{F_{c,a}^{\pm n}(\bfx)\}_{n \in \bbN} \text{ is bounded}\};
$$
the intersection of these sets is denoted by
$$
K(F_{c,a}) := K^+(F_{c,a}) \cap K^-(F_{c,a}).
$$
The forward and backward Julia sets are 
$$
J^\pm(F_{c,a}) := \partial K^\pm(F_{c,a}),$$
and the Julia set is defined as
$$
J(F_{c,a}) := J^+(F_{c,a}) \cap J^-(F_{c,a}).
$$

Let $\cS(F_{c,a})$ be the set of saddle periodic points of $F_{c,a}$. For $\bfs \in \cS(F_{c,a})$, we denote its stable and unstable manifold by $W^s(\bfs)$ and $W^u(\bfs)$ respectively. An intersection point $\bfq \in W^s(\bfs_1) \cap W^u(\bfs_2)$ with $\bfs_1, \bfs_2 \in \cS(F_{c,a})$ is said to be {\it homoclinic} if $\bfs_1 = \bfs_2$, and {\it heteroclinic}  otherwise.

In dimension one, it is well-known that the set of repelling periodic points is dense in the Julia set of a rational map. This motivates the following alternate notion of the Julia set for $F_{c,a}$:
$$
J^*(F_{c,a}) := \overline{\cS(F_{c,a})} \subset J(F_{c,a}).
$$
It is an open problem (posed by J.H. Hubbard) whether $J^*(F_{c,a}) = J(F_{c,a})$.

\subsection{J-stability}

Consider a holomorphic family $(F_\lambda)_{\lambda \in \Lambda}$ of dissipative H\'enon maps $F_\lambda = F_{c_\lambda, a_\lambda}$ (both $\lambda\mapsto c_\lambda$ and $\lambda\mapsto a_\lambda$ are holomorphic functions)  over a connected complex manifold $\Lambda$. We say that this family is {\it $J$-stable} if its members are topologically conjugate on $J$, and the conjugating map depends continuously on $\lambda$.

Recall that an {\it (unbranched) holomorphic motion} of a set $X \subset \bbC^d$ over $\Lambda$ is a family of mappings $\phi_\lambda : X \to \bbC^d$ such that
\begin{enumerate}[i)]
\item $\phi_{\lambda_0} \equiv \Id_X$ for some $\lambda_0 \in \Lambda$;
\item $\lambda \mapsto \phi_\lambda(x)$ is holomorphic for a fixed $x \in X$; and
\item $x \mapsto \phi_{\lambda}(x)$ is injective for a fixed $\lambda \in \Lambda$.
\end{enumerate}
If a family $\phi_\lambda$ satisfies i) and ii), but not iii), we say that it is a {\it branched holomorphic motion}. An unbranched or branched holomorphic motion is said to be {\it equivariant} if
$$
\phi_\lambda \circ F_{\lambda_0}(x) = F_\lambda \circ \phi_\lambda(x)
\hspace{5mm} \text{for} \hspace{5mm}
x \in X.
$$

The true power of holomorphic motions lies in their extension properties. For $d=1$, we have the following classical result of Ma\~n\'e-Sad-Sullivan and Lyubich \cite{MSS,Lyu}:

\begin{thm}[\bf $\lambda$-lemma]
An unbranched holomorphic motion of $X \subset \bbC$ extends continuously to an unbranched holomorphic motion of $\overline{X}$.
\end{thm}

Holomorphic motions were used in \cite{MSS,Lyu} for studying $J$-stability in dimension one. Namely, suppose the repelling periodic points persist in a holomorphic family of rational maps. Then these points move holomorphically. Using the $\lambda$-lemma, this motion can be extended to the entire Julia set--thereby supplying the required continuously varying conjugacy map for structural stability.

In \cite{DuLy}, Dujardin and Lyubich carry out this approach in a higher dimensional setting. The principal difficulty there is that for $d \geq 2$, the $\lambda$-lemma is no longer true: an unbranched holomorphic motion of a set may not extend to its closure; and even if it does, the resulting motion of the closure may become branched. To allow for such a possibility, they define the following alternative notion of stability.

\begin{defn}[\cite{DuLy}]
A holomorphic family $(F_\lambda)_{\lambda \in \Lambda}$ is {\it weakly $J^*$-stable} if there exists an equivariant branched holomorphic motion of $J^*(F_\lambda)$. 

A map $F_{\lambda_1}$ and its corresponding parameter $\lambda_1 \in \Lambda$ are said to be {\it weakly $J^*$-stable} if $(F_\lambda)_{\lambda \in U}$ is weakly $J^*$-stable for some neighborhood $U \subset \Lambda$ of $\lambda_1$; otherwise, we say that $(F_\lambda)_{\lambda \in \Lambda}$ {\it bifurcates} at $\lambda_1$.
\end{defn}

\begin{thm}[\cite{DuLy}]\label{weak stability}
Let $(F_\lambda)_{\lambda \in \Lambda}$ be a holomorphic family of dissipative H\'enon maps. The following are equivalent:
\begin{enumerate}[i)]
\item The family $(F_\lambda)_{\lambda \in \Lambda}$ is weakly $J^*$-stable.
\item The set $J^*(F_\lambda)$ moves continuously in the Hausdorff topology.
\item The type of every periodic point does not change (saddle, attracting, semi-indifferent).
\item Every saddle point moves under an equivariant unbranched holomorphic motion.
\item Every homoclinic or heteroclinic intersection point moves under an equivariant unbranched holomorphic motion. Consequently, any homoclinic or heteroclinic tangency must persist in the whole family.
\end{enumerate}
\end{thm}

\begin{rem}
It is easy to see that structural $J$-stability (or even just structural $J^*$-stability) implies weak $J^*$-stability. However, it is not known whether the converse is true.
\end{rem}

\subsection{Stability in the attracting and semi-parabolic family}

A H{\'e}non map $F_{c,a}$ is determined uniquely by the multipliers $\mu$ and $\nu$ at a fixed point. In particular, we have:
$$
a = \mu\nu
\hspace{5mm} \text{and} \hspace{5mm}
c=(1+\mu\nu)\left(\frac{\mu}{2}+\frac{\nu}{2}\right)-\left(\frac{\mu}{2}+\frac{\nu}{2}\right)^{2}.
$$
When convenient, we will write $F_{\mu,\nu}$ instead of $F_{c,a}$ to denote a H\'enon map. Note that we have
$$
F_{\mu, \nu}(x,y) \to (f_\mu(x), x)
\hspace{5mm} \text{as} \hspace{5mm}
\nu \to 0,
$$
where $f_\mu$ denotes the quadratic polynomial that has a fixed point of multiplier $\mu$.

Suppose $|\nu| < 1$. If $|\mu| \leq 1$, then $F_{\mu, \nu}$ can be classified into one of four types:
\begin{enumerate}[i)]
\item {\it attracting} if $|\mu| <1$;
\item {\it semi-parabolic} if $\mu = e^{2\pi \frac{p}{q}i}$ for some $p/q \in \bbQ / \bbZ$;
\item {\it semi-Siegel} if $\mu = e^{2\pi \theta i}$ for some $\theta \in (\bbR \setminus \bbQ)/\bbZ$, and $F_{\mu, \nu}$ is locally linearizable at the fixed point; or
\item {\it semi-Cremer} if $\mu = e^{2\pi \theta i}$ for some $\theta \in (\bbR \setminus \bbQ)/\bbZ$, and $F_{\mu, \nu}$ is not locally linearizable at the fixed point.
\end{enumerate}

The dynamics of attracting or semi-parabolic H\'enon maps are well-understood. In particular, the following results hold:

\begin{thm}[\cite{FoSi,HuOV}]
Let $\mu_0 \in \bbD \setminus \{0\}$. Then there exists $\bar\epsilon >0$ such that the holomorphic family $(F_{\mu_0, \nu})_{\nu \in \bbD_{\bar\epsilon} \setminus \{0\}}$ of attracting H\'enon maps is structurally $J$-stable.
\end{thm}

\begin{thm}[\cite{RaTa}]
Let $\mu_0 = e^{i\pi\frac{p}{q}}$ for some $p/q \in \bbQ / \bbZ$. Then there exists $\bar\epsilon >0$ such that the holomorphic family $(F_{\mu_0, \nu})_{\nu \in \bbD_{\bar\epsilon} \setminus \{0\}}$ of semi-parabolic H\'enon maps is structurally $J$-stable.
\end{thm}

\begin{rem}
In fact, for both attracting and semi-parabolic H\'enon maps, an explicit topological model of $J$ (as well as $J^+$) can be provided. For more details, see \cite{FoSi}, \cite{HuOV} and \cite{RaTa}.
\end{rem}

\subsection{The golden-mean semi-Siegel H\'enon family}

Let $\theta_* := (\sqrt{5}-1)/2$ be the inverse golden mean rotation number. Let $F_a = F_{\mu_*, \nu}$ be the H\'enon map such that
$$
\Jac F_a \equiv a \in \bbD\setminus \{0\},
$$
and $F_a$ has a fixed point $\bfx_a \in \bbC^2$ of multipliers $\mu_* := e^{2\pi \theta_* i}$ and $\nu = a/ \mu_*$. Since $\theta_*$ is of bounded type, $F_a$ is locally linearizable at $\bfx_a$, and thus, $\bfx_a$ is semi-Siegel. The holomorphic family $(F_a)_{\nu \in \bbD \setminus \{0\}}$ is referred to as the {\it golden-mean semi-Siegel H\'enon family}.

By Siegel's Theorem \cite{Sieg}, there exists a local biholomorphic change of coordinates $\psi_a : (U, (0,0)) \to (V, \mathbf{x}_a)$ such that 
$$
\psi_a^{-1} \circ F_a \circ \psi_a(x,y) = (\mu_* x, \nu y).
$$
Then $\psi_a$ can be biholomorphically extended to a map $\psi_a : (\mathbb{D} \times \bbC, (0,0)) \rightarrow (\mathcal{C}_a, \mathbf{x}_a)$ so that its image $\mathcal{C}_a := \psi_a(\mathbb{D} \times \bbC)$ is maximal (see \cite{MoNiTaUe}). We call $\mathcal{C}_a$ the {\it Siegel cylinder} of $F_a$. 
Clearly, the orbit of every point in $\mathcal{C}_a$ converges to the biholomorphically embedded disk $\mathcal{D}_a:=\psi_a(\mathbb{D} \times\{0\})$. We call $\mathcal{D}_a$ the {\it Siegel disk} of $F_a$. See \figref{fig:siegelcylinder}.

\begin{figure}[h]
\centering
\includegraphics[scale=0.3]{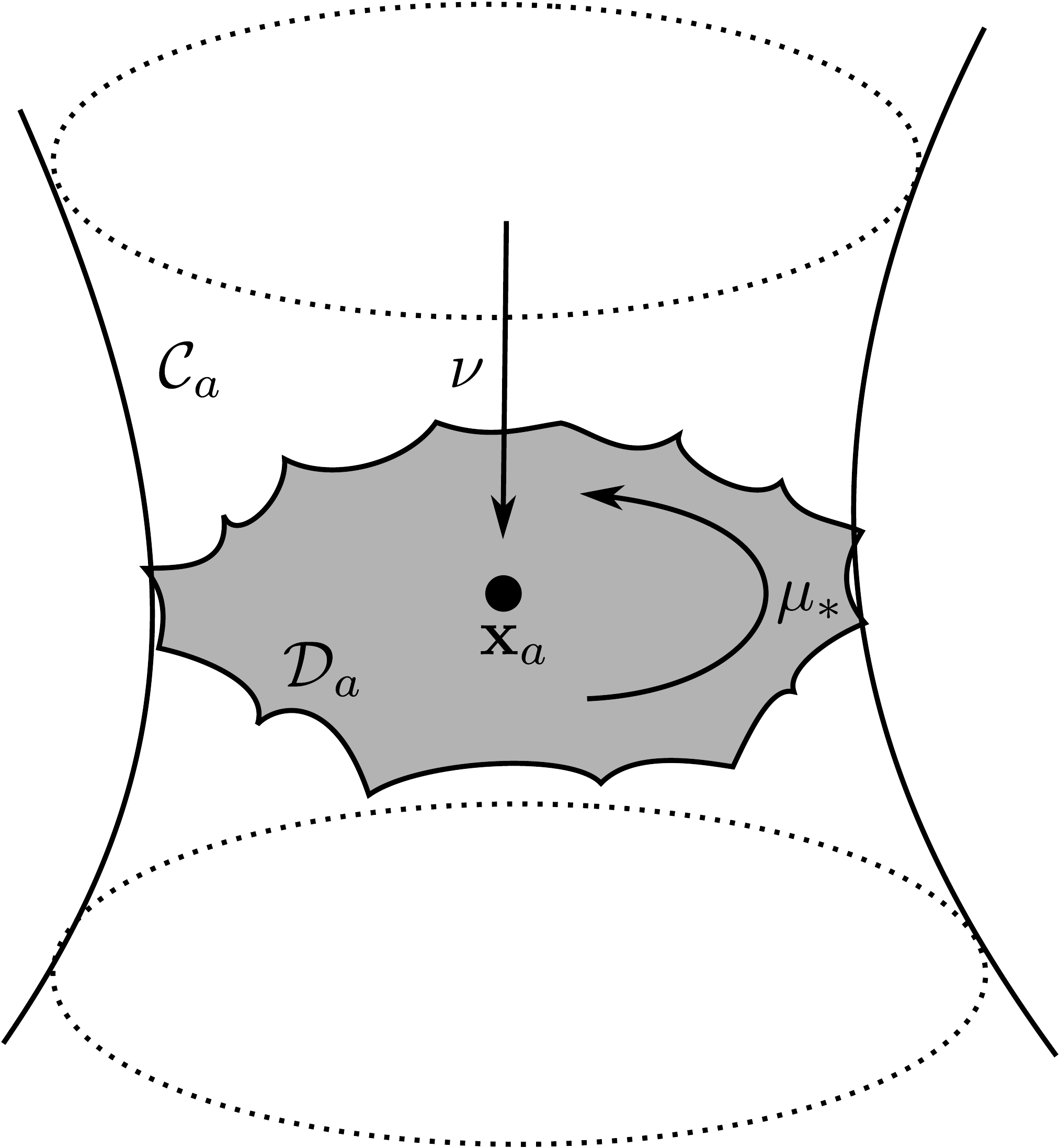}
\caption{The Siegel cylinder $\mathcal{C}_a$ and the Siegel disk $\mathcal{D}_a$ of $F_a$.}
\label{fig:siegelcylinder}
\end{figure}

\begin{rem}
The Siegel cylinder $\cC_a$ is a connected component of the interior of $K^+(F_a)$. Moreover, we have $\partial \cC_a = J^+(F_a)$ (see \cite{BeSm}). The Siegel disk $\cD_a$ is contained in $K$. 
\end{rem}

The first author and D. Gaidashev  developed a renormalization theory of Siegel dynamics that extended to higher dimensions (see \cite{GaYam}). By applying this new tool, jointly with Radu, they proved the following result:

\begin{thm}[\cite{GaRaYam}]\label{continuous circle}
There exists $\bar \epsilon>0$ such that if $b \in \bbD_{\bar\epsilon} \setminus \{0\}$, then the boundary of the Siegel disk $\cD_a$ for $F_a$ is a homeomorphic image of the circle. In fact, the linearizing map
$$
\psi_a|_{y=0} : \bbD \times \{0\} \rightarrow \cD_a
$$
extends continuously and injectively to the boundary.
\end{thm}

\begin{rem}
The above theorem implies that the boundary of the Siegel disk $\partial \cD_a$ is the support of an ergodic invariant measure, and hence, by \cite{Du}, it must be contained in $J^*(F_a)$\footnote{This was pointed out to us by R. Dujardin.}.
\end{rem}

In \cite{Yan1}, the second author reformulated the renormalization theory of \cite{GaYam} to obtain precise quantitative estimates. A summary of this work is given in \secref{sec:prelim}. These estimates then led to the following results on the geometry of the Siegel boundary for H\'enon maps.

\begin{thm}[\cite{Yan1}]
There exists $\bar \epsilon>0$ such that if $a_1, a_2 \in \bbD_{\bar \epsilon}\setminus \{0\}$ and $|a_1| \neq |a_2|$, then $F_{a_1}$ and $F_{a_2}$ cannot be $C^1$-conjugate on the boundary of their respective Siegel disks.
\end{thm}

\begin{thm}[\cite{Yan2}]
There exists $\bar \epsilon>0$ such that the set of parameter values $a$ for which the boundary of the Siegel disk $\cD_a$ for $F_a$ has unbounded geometry contains a dense $G_\delta$ subset in the disc $\bbD_{\bar\epsilon} \setminus \{0\}$.
\end{thm}

The results mentioned thus far are restricted to the boundary of the Siegel disk $\cD_a$. In this paper, we extend our scope to the global behavior of $F_a$ by studying the topological properties of the boundary of the entire Siegel cylinder $\cC_a$. Our main result is the following.

\begin{mainthm}
There exists $\bar\epsilon>0$ such that the golden-mean semi-Siegel H\'enon family $F_a$ is not weakly $J^*$-stable at every parameter $a \in \bbD_{\bar\epsilon} \setminus \{0\}$.
\end{mainthm}

Combining our main theorem with the work of R. Dujardin and M. Lyubich, we are able to conclude the following two surprising facts about the semi-Siegel H\'enon family. First, by Corollary 4.5 in \cite{DuLy}:

\begin{cor}
There exists a dense $G_\delta$ subset $X \subset \bbD_{\bar\epsilon} \setminus \{0\}$ such that $F_a$ is a {\it Newhouse automorphism} (i.e. has infinitely many sinks) for $a \in X$.
\end{cor}

Second, by Theorem 5.7 in \cite{DuLy}:

\begin{cor}
There exists a dense subset $Y \subset \bbD_{\bar\epsilon} \setminus \{0\}$ such that $J(F_a)$ is disconnected for $a \in Y$.
\end{cor}

The strategy of the proof of the Main Theorem is as follows. Using renormalization, we obtain a precise description of the geometry of stable and unstable manifolds of saddles near a dynamically defined point on the boundary of the Siegel disc called the ``fold''. It turns out that under a change in the Jacobian $a$, the stable manifolds, which are nearly vertical, remain almost stationary while  the unstable manifolds, which are ``bent'',  move analytically. Using this fact, we force a contradiction with \thmref{weak stability} v).

Our approach was greatly inspired by the work of Lyubich and Martens, who used renormalization to prove a global topological instability in the family of {\it real} Feigenbaum H\'enon maps (see \cite{LyMa}).

\section{Preliminaries}\label{sec:prelim}

In this section, we provide a brief summary of the renormalization theory of semi-Siegel H\'enon maps. See \cite{Yan1} for complete details.

Let $a \in \bbD_{\bar \epsilon} \setminus \{0\}$ for some $\bar\epsilon>0$ sufficiently small. Consider the H\'enon map
$$
F_a(x,y)=\begin{bmatrix}
x^2+c_a - ay\\
x
\end{bmatrix}
$$
that has a Siegel disc $\cD_a \subset \bbC^2$ with rotation number $\theta_* = (\sqrt{5}-1)/2$. Representing $\theta_*$ by an infinite continued fraction with positive terms, we have:
$$
\theta_* = \cfrac{1}{1+\cfrac{1}{1+\ldots{}}}.
$$
The {\it $n$th partial convergent of $\theta_*$} is the rational number $p_n/q_n$ obtained by terminating the continued fraction of $\theta_*$ after the $n$th term. The denominators $\{q_n\}_{n=0}^\infty$ form the Fibonacci sequence:
\begin{equation}\label{eq:fibonacci}
q_0 = 1
\hspace{5mm} , \hspace{5mm}
q_1 = 1
\hspace{5mm} \text{and} \hspace{5mm}
q_{n+1} = q_n + q_{n-1}
\hspace{5mm} \text{for} \hspace{5mm}
n \geq 1.
\end{equation}

\subsection{Definition of renormalization}

Let $\hat \Omega_0$ and $\hat \Gamma_0$ be suitably chosen topological bidisks in $\bbC^2$ such that $\hat \Omega_0 \cap \hat \Gamma_0 \ni (0,0)$ and $\hat \Omega_0 \cup \hat \Gamma_0 \supset \partial \cD_a$.  The {\it pair representation} of $F_a$ is given by
$$
\hat \Sigma_0 = (\hat A_0, \hat B_0) := (F_a|_{\hat \Omega_0}, F_a|_{\hat \Gamma_0}).
$$

Let
\begin{equation}\label{eq:initial normalization}
\Phi_0(x,y) := (\lambda_0 x, \lambda_0 y),
\end{equation}
where $\lambda_0 := c_a$. Observe that
$$
\Phi_0^{-1} \circ F_a \circ \Phi_0(0,0) =(1, 0).
$$
The {\it normalized pair representation} of $F_a$ is defined as
$$
\Sigma_0 = (A_0, B_0) := (\Phi_0^{-1} \circ \hat A_0 \circ \Phi_0, \Phi_0^{-1} \circ \hat B_0 \circ \Phi_0).
$$

The $n$th {\it renormalization} of $F_a$:
$$
\bfR^n(F_a) := \Sigma_n = (A_n, B_n),
$$
where
$$
A_n(x,y) = \begin{bmatrix}
a_n(x,y) \\
h_n(x,y)
\end{bmatrix}
\hspace{5mm} \text{and} \hspace{5mm}
B_n(x,y) = \begin{bmatrix}
b_n(x,y) \\
x
\end{bmatrix},
$$
is the pair of rescaled iterates of $\Sigma_0$ defined inductively as follows. Denote
$$
(a_n)_y(x) := a_n(x,y),
$$
and let
$$
H_{n+1}(x,y) :=
\begin{bmatrix}
(a_n)_y^{-1}(x) \\
y
\end{bmatrix}.
$$
Consider the non-linear change of coordinates defined as
$$
\Phi_{n+1}(x,y) := H_{n+1} \circ \Lambda_{n+1},
$$
where
\begin{equation}\label{eq:affine rescaling}
\Lambda_{n+1}(x,y) = (\lambda_{n+1} x + c_{n+1}, \lambda_{n+1} y)
\end{equation}
is an affine rescaling map to be specified later. The pair $\Sigma_{n+1} = (A_{n+1}, B_{n+1})$ is defined as
$$
A_{n+1} = \Phi_{n+1}^{-1} \circ B_n \circ A_n^2 \circ \Phi_{n+1}
\hspace{5mm} \text{and} \hspace{5mm}
B_{n+1} = \Phi_{n+1}^{-1} \circ B_n \circ A_n \circ \Phi_{n+1}.
$$

\begin{lem}\label{definition}
There exist topological discs $Z, W, V \subset \bbC$ containing $0$ such that the following holds. Denote
$$
V_n := \lambda_1^{-1} \cdot \ldots \cdot \lambda_n^{-1} V
\hspace{5mm} , \hspace{5mm}
\Omega_n := Z \times V_n
\hspace{5mm} \text{and} \hspace{5mm}
\Gamma_n := W \times V_n.
$$
Then for all $n \in \bbN$ sufficiently large, $A_n$ and $B_n$ are bounded analytic maps that are well-defined on $\Omega_n$ and $\Gamma_n$ respectively.
\end{lem}

\begin{lem}\label{degeneration}
The dependence of $\Sigma_n = (A_n, B_n)$ on $y$ decays \emph{super-exponentially} fast. That is, we have
$$
\sup_{(x,y) \in \Omega}\|\partial_y A_n(x,y)\| < C\bar\epsilon^{2^n}
\hspace{5mm} \text{and} \hspace{5mm}
\sup_{(x,y) \in \Gamma}\|\partial_y B_n(x,y)\| < C\bar\epsilon^{2^n}.
$$
for some uniform constant $C >0$.
\end{lem}

The one-dimensional projections of $A_n$, $B_n$ and $\Sigma_n$ are given by
\begin{equation}\label{eq:projections}
\eta_n(x) := a_n(x, 0)
\hspace{5mm} , \hspace{5mm}
\xi_n(x) := b_n(x,0)
\hspace{5mm} \text{and} \hspace{5mm}
\zeta_n := (\eta_n, \xi_n)
\end{equation}
respectively. By \lemref{definition}, we see that $\eta_n$ and $\xi_n$ are bounded analytic functions defined on $Z$ and $W$ respectively. Moreover, by \lemref{degeneration}, the dynamics of $\Sigma_n$ degenerates to that of $\zeta_n$ super-exponentially fast. It is shown in \cite{Yan1} that $\eta_n$ and $\xi_n$ each have a unique simple critical point which are $C\bar\epsilon^{2^n}$-close to each other. We choose the normalizing constants $\lambda_n$ and $c_n$ in \eqref{eq:affine rescaling} so that
$$
\xi_n(0) =1
\hspace{5mm} \text{and} \hspace{5mm}
\xi_n'(0) = 0.
$$

\subsection{Renormalization convergence}

Let $\zeta_* = (\eta_*, \xi_*)$ be the fixed point of the one-dimensional renormalization operator $\cR$ given in \cite{GaYam}. In particular, we have
\begin{equation}\label{eq:1d limit}
\lambda_*^{-1} \eta_* \circ \xi_* \circ \eta_* (\lambda_* x) = \eta_*(x)
\hspace{5mm} \text{and} \hspace{5mm}
\lambda_*^{-1} \eta_* \circ \xi_* (\lambda_* x)=\xi_*(x),
\end{equation}
where
$$
\lambda_* := \eta_* \circ \xi_* (0) \in \bbD
$$
is the universal scaling factor.

Convergence under renormalization for semi-Siegel H\'enon maps was first obtained in \cite{GaYam}. For the renormalization operator $\bfR$ defined above, the proof of convergence is given in \cite{Yan1}.

\begin{thm}\label{convergence}
As $n \to \infty$, we have the following convergences (each of which occurs at an exponential rate):
\begin{enumerate}[(i)]
\item $\zeta_n =(\eta_n, \xi_n) \to \zeta_* = (\eta_*, \xi_*)$;
\item $\lambda_n \to \lambda_*$; and
\item $\Phi_n \to \Phi_*$, where 
$$
\Phi_*(x,y) :=
\begin{bmatrix}
\phi_*(x) \\
\lambda_* y
\end{bmatrix}=
\begin{bmatrix}
\eta_*^{-1}(\lambda_* x) \\
\lambda_* y
\end{bmatrix}.
$$
\end{enumerate}
\end{thm}

\subsection{First return iterate}

Define the {\it $n$th microscope map of depth $k$} by
$$
\Phi_n^{n+k} := \Phi_{n+1} \circ \Phi_{n+2} \circ \ldots{} \circ \Phi_{n+k}.
$$
Let
$$
\Omega_n^{n+k} := \Phi_n^{n+k}(\Omega_{n+k})
\hspace{5mm} \text{and} \hspace{5mm}
\Gamma_n^{n+k} := \Phi_n^{n+k}(\Gamma_{n+k}).
$$
Observe that $\{\Omega_n^{n+k} \cup \Gamma_n^{n+k}\}_{k=0}^\infty$ is a nested sequence of open sets. See \figref{fig:microscope}.

\begin{prop}\label{2dcap}
Let $\lambda_* \in \bbD$ be the universal scaling factor. Then we have
$$
\|D\Phi_n^{n+k}\| \asymp |\lambda_*|^k.
$$
Consequently, there exists a point $(\kappa_n, 0) \in \bbC^2$, called the \emph{$n$th fold}, such that
$$
\bigcap_{k =0}^\infty \Phi_n^{n+k}((Z \cup W) \times V) = \{(\kappa_n,0)\}.
$$
\end{prop}

\begin{rem}
The fold is a dynamically defined point with the same combinatorial address as the critical value $\xi_*(0) = 1$ for the one-dimensional renormalization limit $\zeta_* = (\eta_*, \xi_*)$. In \cite{dCLyMa}, its analog is referred to as the {\it tip}.
\end{rem}

\begin{figure}[h]
\centering
\includegraphics[scale=0.25]{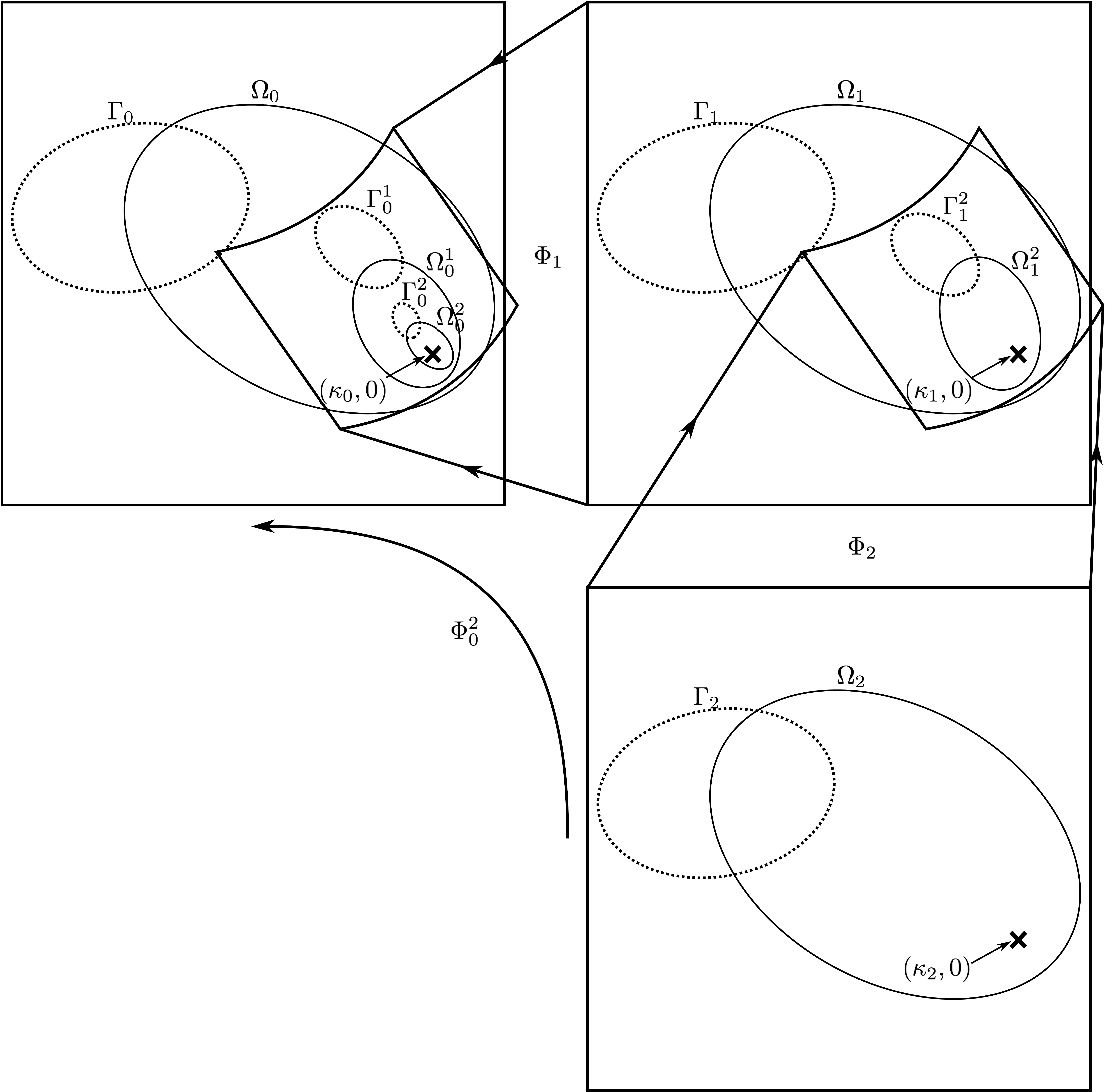}
\caption{The renormalization microscope map $\Phi_0^2$ obtained by composing the non-linear changes of coordinates $\Phi_1$ and $\Phi_2$. We have $\Omega_0^1 = \Phi_1(\Omega_1)$, $\Gamma_0^1 = \Phi_1(\Gamma_1)$, $\Omega_0^2 =\Phi_0^2(\Omega_2)$, $\Gamma_0^2 = \Phi_0^2(\Gamma_2)$, and $(\kappa_0,0)=\Phi_1(\kappa_1,0) = \Phi_0^2(\kappa_2,0)$.}
\label{fig:microscope}
\end{figure}

We denote by
$$
p\Sigma_n^{n+k} = (pA_n^{n+k}, pB_n^{n+k})
\hspace{5mm} \text{for} \hspace{5mm}
n, k \geq 0
$$
the sequence of pairs of iterates of $\Sigma_n=(A_n,B_n)$ defined inductively as follows:
\begin{enumerate}[(i)]
\item $p\Sigma_n^n := \Sigma_n$, and
\item $p\Sigma_n^{n+k+1} := (pB_n^{n+k} \circ (pA_n^{n+k})^2, pB_n^{n+k} \circ pA_n^{n+k})$.
\end{enumerate}
It is not hard to see that we have
\begin{equation}\label{eq:prerenorm}
A_{n+k} = (\Phi_n^{n+k})^{-1} \circ pA_n^{n+k} \circ \Phi_n^{n+k}
\hspace{5mm} \text{and} \hspace{5mm}
B_{n+k} = (\Phi_n^{n+k})^{-1} \circ pB_n^{n+k} \circ \Phi_n^{n+k}.
\end{equation}
Moreover, by \lemref{definition}, $pA_n^{n+k}$ and $pB_n^{n+k}$ are well-defined on $\Omega_n^{n+k}$ and $\Gamma_n^{n+k}$ respectively.

Let
$$
\hat\Omega_0^n := \Phi_0(\Omega_0^n)
\hspace{5mm} \text{and} \hspace{5mm}
\hat\Gamma_0^n := \Phi_0(\Gamma_0^n),
$$
where $\Phi_0$ is given in \eqref{eq:initial normalization}. By construction, the $n$th renormalization $$\bfR^n(F_a) := \Sigma_n = (A_n, B_n)$$ of $F_a$ represents the first return map under $F_a$ to the neighborhood $\hat\Omega_0^n \cup \hat\Gamma_0^n$. The following lemma is a formal statement of this fact.

\begin{lem}\label{renorm iterate}
Let $\{q_n\}_{n=0}^\infty$ be the Fibonacci sequence defined in \eqref{eq:fibonacci}. Define
$$
p\hat A_0^n := F_a^{q_{2n+1}}|_{\hat \Omega_0^n}
\hspace{5mm} \text{and} \hspace{5mm}
p\hat B_0^n := F_a^{q_{2n}}|_{\hat \Gamma_0^n}.
$$
Then
$$
pA_0^n = \Phi_0^{-1} \circ p\hat A_0^n \circ \Phi_0
\hspace{5mm} \text{and} \hspace{5mm}
pB_0^n = \Phi_0^{-1} \circ p\hat B_0^n \circ \Phi_0.
$$
Thus, $\Sigma_n = (A_n, B_n)$ is given by
$$
A_n = (\Phi_0^n)^{-1} \circ \Phi_0^{-1} \circ p\hat A_0^n \circ \Phi_0 \circ \Phi_0^n
\hspace{5mm} \text{and} \hspace{5mm}
B_n = (\Phi_0^n)^{-1} \circ \Phi_0^{-1} \circ p\hat B_0^n \circ \Phi_0 \circ \Phi_0^n.
$$
\end{lem}

\subsection{Limit of microscope maps}

Let
$$
\phi_n(x) := \Phi_n(x,0)
\hspace{5mm} \text{and} \hspace{5mm}
\phi_n^{n+k}(x) := \Phi_n^{n+k}(x,0).
$$
It is not difficult to see that
$$
\phi_n^{n+k}= \phi_{n+1} \circ \phi_{n+2} \ldots \phi_{n+k}.
$$
Furthermore, we have
$$
\phi_n(\kappa_n) = \kappa_{n-1}
\hspace{5mm} \text{and} \hspace{5mm}
\phi_n^{n+k}(\kappa_{n+k}) = \kappa_n,
$$
where $(\kappa_n, 0)$ is the $n$th fold given in \propref{2dcap}. Denote
$$
d_n := \phi_n'(\kappa_n)
\hspace{5mm} \text{and} \hspace{5mm}
d_n^{n+k} := (\phi_n^{n+k})'(\kappa_{n+k}) = d_{n+1} d_{n+2} \ldots d_{n+k}.
$$

\begin{prop}
The map $\phi_* : Z\cup W \to Z$ given in \thmref{convergence} has an attracting fixed point at $1$ with multiplier $\lambda_*^2$.
\end{prop}

\begin{prop}\label{limit phi x}
As $n \to \infty$, we have the following convergences (each of which occurs at an exponential rate):
\begin{enumerate}[i)]
\item $\phi_n \to \phi_*$;
\item $\kappa_n \to 1$; and
\item $d_n \to \lambda_*^2$.
\end{enumerate}
\end{prop}

Let
\begin{equation}\label{eq:check phi x}
\check{\phi}_n^{n+k}(x) := \phi_n^{n+k}(x+\kappa_{n+k}) - \kappa_n
\hspace{5mm} \text{and} \hspace{5mm}
\check{\phi}_*(x) := \phi_*(x+1) -1.
\end{equation}
Observe that $\check{\phi}_n^{n+k}$ and $\check{\phi}_*$ both have an attracting fixed point at $0$ of multiplier $d_n^{n+k}$ and $\lambda_*^2$ respectively.

\begin{prop}\label{limit micro}
We have the following convergence (which occurs at an exponential rate):
$$
(d_n^{n+k})^{-1} \check\phi_n^{n+k} \to u_*
\hspace{5mm} \text{as} \hspace{5mm}
k \to \infty,
$$
where $u_*$ is the linearizing map for $\check\phi_*$ at $0$.
\end{prop}

\subsection{Universality}

In \cite{Yan1}, the second author showed that the sequence of maps $$\bfR^n(F_a) = \Sigma_n = (A_n, B_n)$$ takes on a {\it universal} two-dimensional shape as it flattens and converges to the degenerate one-dimensional renormalization limit given in \thmref{convergence}:

\begin{thm}\label{universalityb}
For $n \geq 0$, we have
$$
B_n(x,y)
=\begin{bmatrix}
\xi_n(x) - a^{q_{2n}} \, \beta(x) \, y \, (1+O(\rho^n))\\
x
\end{bmatrix},
$$
where $0 < \rho <1$ is a uniform constant; $\{q_n\}_{n=0}^\infty$ is the Fibonacci sequence; and $\beta(x)$ is a universal function that is uniformly bounded away from $0$ and $\infty$, and has a uniformly bounded derivative and distortion.
\end{thm}

\begin{thm}\label{universalitya}
For $n \geq 0$, we have
$$
A_n(x,y) =
\begin{bmatrix}
\eta_n(x) - a^{q_{2n}} \, \alpha(x) \, y \, (1+O(\rho^n))\\
\lambda_n^{-1} \tilde \eta_{n-1}(\lambda_n x) - a^{q_{2n}} \, \tilde{\alpha}(x) \, y \, (1+O(\rho^n))
\end{bmatrix},
$$
where $0 < \rho <1$ is a uniform constant; $\{q_n\}_{n=0}^\infty$ is the Fibonacci sequence; and
$$
\tilde \eta_k(x) = \eta_k(x) + O(a^{q_{2k}})
\hspace{3mm} , \hspace{3mm}
\alpha(x) := \frac{\eta_*'(x)}{\xi_*'(x)}\beta(x)
\hspace{3mm} \text{and} \hspace{3mm}
\tilde{\alpha}(x) := \frac{\eta_*'(\lambda_* x)}{\eta_*'(x)} \alpha(x).
$$
\end{thm}

Recall that we have
$$
\Phi_n(\kappa_n, 0) = (\kappa_{n-1}, 0)
\hspace{5mm} \text{and} \hspace{5mm}
\Phi_n^{n+k}(\kappa_{n+k}, 0) = (\kappa_n, 0).
$$
Denote
$$
D_n:= D_{(\kappa_n,0)}\Phi_n
\hspace{5mm} \text{and} \hspace{5mm}
D_n^{n+k} := D_{(\kappa_{n+k},0)}\Phi_n^{n+k} = D_{n+1} \cdot D_{n+2} \cdot \ldots{} \cdot D_{n+k}.
$$

\begin{thm}\label{cap derivative}
Write
$$
D_n = \begin{bmatrix}
1 & t_na^{q_{2(n-1)}} \\
0 & 1
\end{bmatrix}
\begin{bmatrix}
u_n & 0 \\
0 & \lambda_n
\end{bmatrix}.
$$
Then there exist a uniform positive constant $\rho <1$ such that the following estimates hold for all $n \geq 1$:
\begin{enumerate}[(i)]
\item $u_n = \lambda_*^2(1+O(\rho^n))$,
\item $\lambda_n = \lambda_*(1+O(\rho^n))$, and
\item $t_n = \lambda_*\alpha(1)(1+O(\rho^n))$.
\end{enumerate}
Consequently, for $0 \leq n, k$, we have
$$
D_n^{n+k} =  \begin{bmatrix}
1 & t_n^{n+k}a^{q_{2n}}\\
0 & 1
\end{bmatrix}
\begin{bmatrix}
u_n^{n+k} & 0 \\
0 & \lambda_n^{n+k}
\end{bmatrix},
$$
where
\begin{enumerate}[(i)]
\item $u_n^{n+k} := u_{n+1} \cdot u_{n+1} \cdot \ldots \cdot u_{n+k} = \lambda_*^{2k}(1+O(\rho^n))$,
\item $\lambda_n^{n+k} := \lambda_{n+1} \cdot \lambda_{n+2} \cdot \ldots \cdot \lambda_{n+k} = \lambda_*^k(1+O(\rho^n))$, and
\item $t_n^{n+k} = t_{n+1}(1 +O(a^{q_{2n+1}}))$.
\end{enumerate}
\end{thm}

\section{The Underlying One-Dimensional Dynamics}\label{sec:1d}

Consider the map $\xi_n: W \to \xi_n(W)$ given in \eqref{eq:projections}. By \thmref{convergence}, $\xi_n$ converges exponentially fast to $\xi_*$, where $\zeta_* = (\eta_*, \xi_*)$ is the fixed point of the one-dimensional renormalization operator $\cR$.

\begin{prop}\label{fixed1d}
The map $\xi_n$ has a unique repelling fixed point $\check{x}_n \in W$. Moreover, $\check{x}_n$ converges exponentially fast to the unique repelling fixed point $x_* \in W$ for $\xi_*$.
\end{prop}

\begin{proof}
Proposition 3.2 of \cite{Yam} implies that $(\eta_n,\xi_n)$ forms a holomorphic pair in the sense of McMullen, with {\it a priori} bounds. By Denjoy-Wolff Theorem, $\xi_n$ has a unique fixed point in its domain of definition, which is attracting under $\xi_n^{-1}$, and hence is repelling under $\xi_n$. Finally, the exponential convergence statement follows by Proposition~3.1 of \cite{Yam}.
\end{proof}

Let $U_n$ be a small neighborhood of $\check{x}_n$, and denote:
$$
\xi_n^{\exp}:= \xi_n|_{U_n}.
$$
Then $\xi_n^{\exp}$ maps $U_n$ conformally onto $\xi_n^{\exp}(U_n) \Supset U_n$. We may choose a domain $\cU \subset W$ such that for all $n \in \bbN$ sufficiently large, the following hold:
\begin{enumerate}[i)]
\item $\check{x}_n \in \cU$;
\item $(\xi_n^{\exp})^{-1}$ extends to $\cU$; and
\item $\cU \setminus \overline{(\xi_n^{\exp})^{-1}(\cU)}$ is a topological annulus containing $0$.
\end{enumerate}
Note that $|(\xi_n^{\exp})'(x)| > C$ for $x\in (\xi_n^{\exp})^{-1}(\cU)$ for some uniform constant $C >0$.

Let $D_\delta(0)$ be a disc of sufficiently small radius $\delta > 0$ centered at $0$. Recall that $\xi_n : W \to \xi_n(W)$ has a unique simple critical point at $0$, and that $\xi_n(0) = 1$. It follows that $\xi_n$ has two distinct inverse branches on $D_\delta(0)$, one of which is $(\xi_n^{\exp})^{-1}$, and the other we denote by $(\xi_n^{\rot})^{-1}$. Note that we have
$$
(\xi_n^{\exp})^{-1}(D_\delta(0)) \subset (\xi_n^{\exp})^{-1}(\cU)
\hspace{5mm} \text{and} \hspace{5mm}
(\xi_n^{\rot})^{-1}(D_\delta(0)) \cap \cU = \varnothing.
$$
See \figref{fig:branches}.

\begin{rem}[See Prop.~3.1 of \cite{Yam}]\label{convergence of inverses}
As $n \to \infty$, the maps $(\xi_n^{\exp})^{-1}$ and $(\xi_n^{\rot})^{-1}$ converge exponentially fast to two distinct inverse branches $(\xi_*^{\exp})^{-1}$ and $(\xi_*^{\rot})^{-1}$ of $\xi_*$. 
\end{rem}

\begin{figure}[h]
\centering
\includegraphics[scale=0.3]{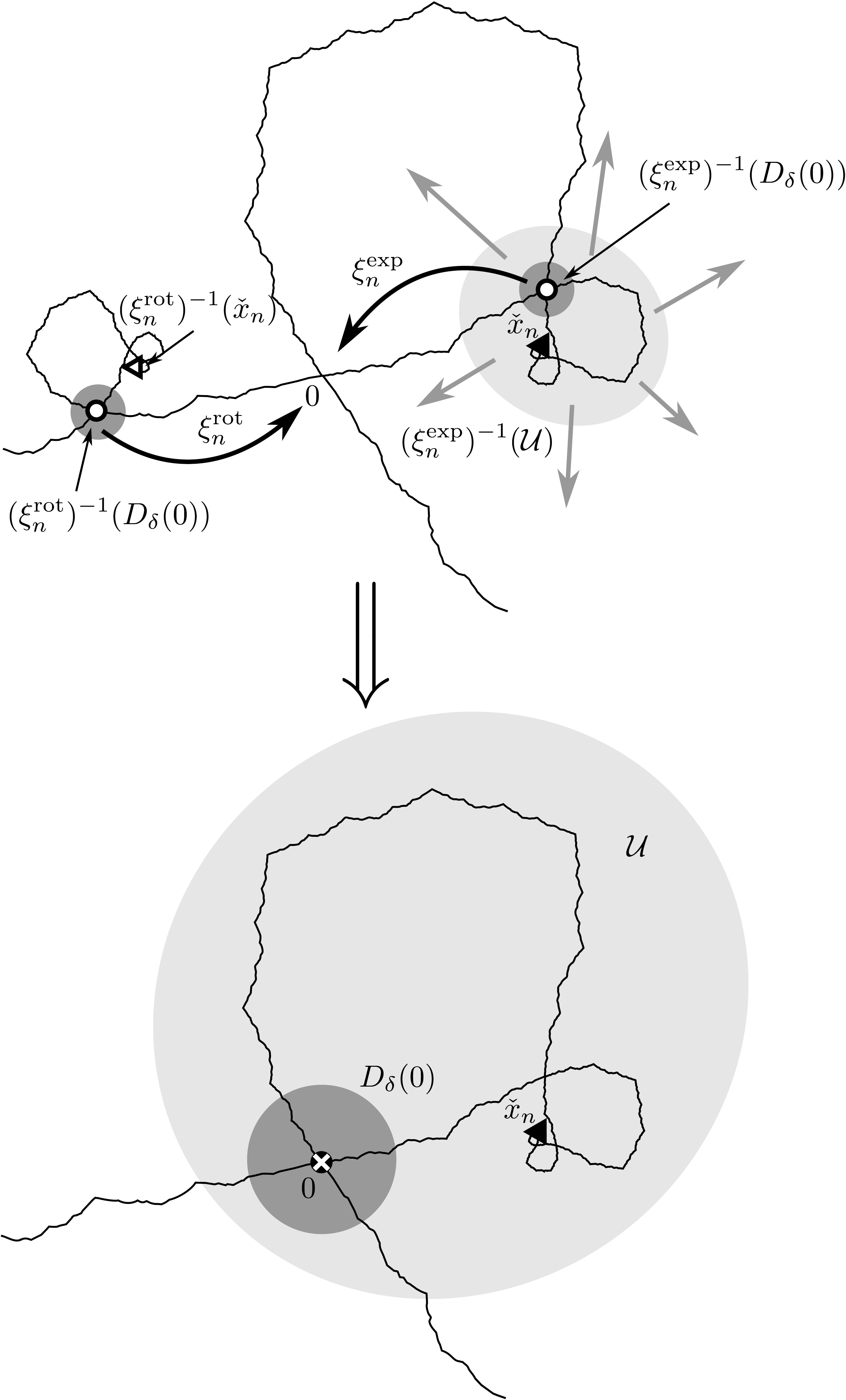}
\caption{A choice of the domain $\cU \subset W$ on which $\xi_n$ is uniformly expanding. The other branch $(\xi_n^{\rot})^{-1}$ maps the disc $D_\delta(0)$ outside of $\cU$.}
\label{fig:branches}
\end{figure}

\begin{thm}[\cite{Yan2}]\label{universalityaphi}
For $n \geq 0$, we have
$$
A_n \circ \Phi_{n+1}(x,y) =
\begin{bmatrix}
\lambda_{n+1}x+c_{n+1}\\
(\tilde{\xi}_n^{\rot})^{-1}(\lambda_{n+1}x) - a^{q_{2n+1}} \, \tilde \beta (x)\, \lambda_{n+1}\, y(1+O(\rho^n))
\end{bmatrix},
$$
where $0 < \rho <1$ is a uniform constant; $\{q_n\}_{n=0}^\infty$ is the Fibonacci sequence; and
$$
(\tilde \xi_n^{\rot})^{-1}(x) = (\xi_n^{\rot})^{-1}(x) + O(a^{q_{2(n-1)}})
\hspace{3mm} \text{and} \hspace{3mm}
\tilde \beta(x) := \frac{\xi_*'(\lambda_* x)}{\xi_*'(x)}\frac{\beta(x)}{\beta(\lambda_* x)}.
$$
\end{thm}

\subsection{The degenerate map in dimension two}

Consider the degenerate map $\check{B}_n : \Gamma_n = W \times V_n \to \check{B}_n(\Gamma_n)$ defined by
$$
\check{B}_n(x,y) := (\xi_n(x), x).
$$
By \propref{fixed1d}, we see that the map $\check{B}_n$ has a unique saddle point 
$$
\check{\bfs}_n := (\check{x}_n, \check{x}_n) \in \Gamma_n \subset \bbC^2
$$
with one multiplier the same as that of $\check{x}_n$ for $\xi_n$, and the other equal to $0$. The {\it local stable manifold} of $\check{\bfs}_n$ can be identified with the vertical line segment through $\check{\bfs}_n$:
$$
\cM^s_{\loc}(\check{\bfs}_n) := \{(\check{x}_n, y) \, | \, y \in V_n\}.
$$
In other words, $\cM^s_{\loc}(\check{\bfs}_n)$ is the graph of the constant map
$$
y \mapsto \check{x}
\hspace{5mm} \text{for} \hspace{5mm}
y \in V_n.
$$
The {\it local unstable manifold} of $\check{\bfs}_n$ can be identified with the following embedding of $(\xi_n^{\exp})^{-1}(\cU) \subset \bbC$ into $\bbC^2$:
$$
\cM^u_{\loc}(\check{\bfs}_n) := \{(\xi_n(x), x) \, | \, x \in (\xi_n^{\exp})^{-1}(\cU)\}.
$$
Observe
$$
\check{B}_n(\cM^s_{\loc}(\check{\bfs}_n)) = \{\check{\bfs}_n\} \subset \cM^s_{\loc}(\check{\bfs}_n)
$$
and
$$
\check{B}_n(\cM^u_{\loc}(\check{\bfs}_n)) = \{(\xi_n(y),y) \, | \, y \in \cU\} \supset \cM^u_{\loc}(\check{\bfs}_n).
$$

For ease of reference, let us formulate the following:
\begin{lem}
The local unstable manifold $\cM^u_{\loc}(\check{\bfs}_n)$ is the graph of the map
$$
x \mapsto (\xi_n^{\exp})^{-1}(x)
\hspace{5mm} \text{for} \hspace{5mm}
x \in \cU.
$$
\end{lem}

\begin{lem}
The degenerate map $\check{B}_n$ is injective on $\cM^u_{\loc}(\check{\bfs}_n)$. The image $\check{B}_n(\cM^u_{\loc}(\check{\bfs}_n))$ is the graph of the map
$$
y \mapsto \xi_n(y)
\hspace{5mm} \text{for} \hspace{5mm}
y \in \cU.
$$
Consequently, $(1,0)$ is the unique point in $\check{B}_n(\cM^u_{\loc}(\check{\bfs}_n))$ that has a vertical tangent. We call $(1,0)$ the \emph{vertical tangency} for $\check{B}_n$.
\end{lem}

\section{Local Stable and Unstable Manifolds}\label{sec:local manifolds}

\subsection{Saddle fixed points}

Consider the $n$th renormalization $\Sigma_n = (A_n, B_n)$ of the semi-Siegel H\'enon map $F_a$. By \thmref{universalityb}, the map $B_n : \Gamma_n = W \times V_n \to B_n(\Gamma_n)$ given by 
$$
B_n(x,y) = (b_n(x,y), x)
$$
is an $O(a^{q_{2n}})$-perturbation of the degenerate map $\check{B}_n$ defined in \secref{sec:1d}. Hence, $B_n$ has a unique saddle fixed point
$$
\bfs_n = (x_n, y_n) \in \Gamma_n \subset \bbC^2
$$
which is $O(a^{q_{2n}})$-close to $\check{\bfs}_n = (\check{x}_n, \check{x}_n)$.

Consider
$$
D_{\bfs_n}B_n = \begin{bmatrix}
\partial_x b_n(\bfs_n) & \partial_y b_n(\bfs_n)\\
1 & 0
\end{bmatrix}.
$$
The eigenvalues of this matrix are
$$
\mu_n \approx \xi_n'(\check x_n)
\hspace{5mm} \text{and} \hspace{5mm}
\nu_n = \partial_y b_n(\bfs_n) = O(a^{q_{2n}}).
$$
The eigendirections corresponding to $\mu_n$ and $\nu_n$ are given by the eigenvectors
\begin{equation}\label{eq:eigenvectors}
\vec u_n \approx \begin{bmatrix}
1\\
1/\xi_n'(\check x_n)
\end{bmatrix}
\hspace{5mm} \text{and} \hspace{5mm}
\vec v = \begin{bmatrix}
0\\
1
\end{bmatrix}.
\end{equation}
respectively.

Let $\cM^s(\bfs_n)$ and $\cM^u(\bfs_n)$ be the stable and unstable manifold of $\bfs_n$ respectively. The {\it local stable manifold} $\cM^s_{\loc}(\bfs_n)$ of $\bfs_n$ is defined as the connected component of $\cM^s(\bfs_n) \cap \Gamma_n$ that contains $\bfs_n$. The {\it local unstable manifold} $\cM^u_{\loc}(\bfs_n)$ of $\bfs_n$ is defined as the connected component of $\cM^u(\bfs_n) \cap (\cU \times V_n)$ containing $\bfs_n$, where $\cU \subset W$ is the domain chosen in \secref{sec:1d}. Observe that the eigenvectors $\vec v$ and $\vec u_n$ given in \eqref{eq:eigenvectors} are tangent to $\cM^s_{\loc}(\bfs_n)$ and $\cM^u_{\loc}(\bfs_n)$ respectively at $\bfs_n$.

\begin{prop}
The local stable manifold $\cM^s_{\loc}(\bfs_n)$ of $\bfs_n$ for $B_n$ is the graph of the map
$$
y \mapsto \psi_n(y)
\hspace{5mm} \text{for} \hspace{5mm}
y \in V_n,
$$
where $\psi_n$ is analytic and $O(a^{q_{2n}})$-close to the constant map $y \mapsto x_n$.
\end{prop}

\begin{proof}
Denote
$$
(b_n)_y(x) := b_n(x,y).
$$
For $a$ sufficiently small, we have
$$
|(b_n)_y'(x_n)| \approx |\xi_n'(\check x_n)| >1
\hspace{5mm} \text{for} \hspace{5mm}
y \in V_n.
$$
Hence, we may choose a small neighborhood $U_n \subset \bbC$ of $x_n$, so that 
$$
|(b_n)_y'(x)| > C >1
\hspace{5mm} \text{for} \hspace{5mm}
x\in U_n
\hspace{2.5mm} , \hspace{2.5mm}
y \in V_n.
$$
Consequently, $(b_n)_y$ maps $U_n$ conformally onto $(b_n)_y(U_n) \Supset U_n$ for $y \in V_n$.

Let $\cG$ be the family of graphs such that if $G_\psi \in \cG$, then $G_\psi$ is the graph of the map
$$
y \mapsto \psi(y)
\hspace{5mm} \text{for} \hspace{5mm}
y \in V_n,
$$
where $\psi : V_n \to U_n$ is analytic. We claim that for every $G_\psi \in \cG$, there exists a unique graph $G_\phi \in \cG$ such that $B_n(G_\phi) \subset G_\psi$. Observe that for $y \in V_n$:
\begin{equation}\label{eq:graph condition}
B_n(x,y) \in G_\psi \iff (b_n)_y(x) = \psi(x).
\end{equation}
Since $(b_n)_y$ conformally maps $U_n$ onto a set that compactly contains $U_n$, while $\psi(V_n) \subset U_n$, there exists a unique value $x = \phi(y) \in U_n$ that satisfies \eqref{eq:graph condition}. The claim follows.

From the claim, we obtain the {\it graph transform} $\cT : \cG \to \cG$ defined by $\cT : \psi \mapsto \phi$. Since $(b_n)_y$ is expanding on $U_n$, we can show by the standard techniques that $\cT$ contracts the $C^0$-distance on $\cG$, and that its unique fixed point is $G_{\psi_n} = \cM^s_{\loc}(\bfs_n)$.

It remains to show that $\|\psi_n'\| = O(a^{q_{2n}})$. Letting $x = \psi_n(y)$ and $\psi = \psi_n$ in \eqref{eq:graph condition} and differentiating with respect to $y$, we obtain
$$
\psi_n'(y) = \frac{\partial_y b_n(x,y)}{\psi_n'(x) - \partial_x b_n(x,y)}.
$$
Since $\psi_n$ maps a large domain $V_n$ into a small domain $U_n$, we have
$$
|\psi_n'(x)| < c < 1 < C < |\partial_x b_n(x,y)|.
$$
The result follows.
\end{proof}

\begin{prop}\label{fixed unstable}
The local unstable manifold $\cM^u_{\loc}(\bfs_n)$ of $\bfs_n$ for $B_n$ is the graph of the map
$$
x \mapsto \chi_n(x)
\hspace{5mm} \text{for} \hspace{5mm}
x \in \cU,
$$
where $\chi_n$ is analytic and $O(a^{q_{2n}})$-close to $(\xi_n^{\exp})^{-1}$.
\end{prop}

\begin{proof}
Let $G_\chi$ be the graph of the map
$$
x \mapsto \chi(x)
\hspace{5mm} \text{for} \hspace{5mm}
x \in U,
$$
where $U$ is a domain in $(\xi_n^{\exp})^{-1}(\cU)$, and $\chi : U \to \cU$ is analytic. We claim that $B_n(G_\chi)$ is the graph of the map
$$
x \mapsto B_n^\chi(x)
\hspace{5mm} \text{for} \hspace{5mm}
x \in U',
$$
where $U' \approx \xi_n^{\exp}(U)$, and $B_n^\chi : U' \to \cU$ is analytic and $O(a^{q_{2n}})$-close to $(\xi_n^{\exp})^{-1}$. Observe
$$
B_n(x', y') = ((b_n)_{y'}(x'), x') = (x, (b_n)_{y'}^{-1}(x)).
$$
Since $(b_n)_{y'}$ is $O(a^{q_{2n}})$-close to $\xi_n^{\exp}$ on $(\xi_n^{\exp})^{-1}(\cU)$ for $y' \in V_n$, the claim follows. 

The eigenvector $\vec u_n$ given in \eqref{eq:eigenvectors} is tangent to $\cM^u_{\loc}(\bfs_n)$ at $\bfs_n$. Thus, there exists a small neighborhood $U_n \subset \bbC$ of $x_n$ such that near $\bfs_n$, the stable manifold $\cM^u_{\loc}(\bfs_n)$ is given by the graph of the map
$$
x \mapsto \chi_n(x)
\hspace{5mm} \text{for} \hspace{5mm}
x \in U_n,
$$
where $\chi_n$ is analytic, and $\chi_n'(x_n) \approx 1/\xi_n'(\check x_n)$. Since $U_n$ maps onto $\cU$ under a finite iterate of $\xi_n$, the result follows from the claim.
\end{proof}

\subsection{Saddle orbits of higher periods}

The phase space $\Omega_{n+k} \cup \Gamma_{n+k}$ for $$\Sigma_{n+k} = (A_{n+k}, B_{n+k})$$ can be identified as the subset $\Omega_n^{n+k} \cup \Gamma_n^{n+k} $ of the phase space $\Omega_n \cup \Gamma_n$ for $$\Sigma_n = (A_n, B_n)$$ via the $n$th microscope map $\Phi_n^{n+k}$ of depth $k$:
$$
\text{i.e.}\hspace{5mm}
\Phi_n^{n+k}(\Omega_{n+k} \cup \Gamma_{n+k})=\Omega_n^{n+k} \cup \Gamma_n^{n+k} \subset \Omega_n \cup \Gamma_n.
$$
Let $\bfs_{n+k} \in \Gamma_{n+k}$ be the unique saddle fixed point for $B_{n+k}$. Denote
$$
\bfs_n^{n+k} = (x_n^{n+k}, y_n^{n+k}):= \Phi_n^{n+k}(\bfs_{n+k}).
$$
Observe that $\bfs_n^{n+k}$ is a saddle fixed point for the iterate $pB_n^{n+k}$ of $\Sigma_n$ (see \eqref{eq:prerenorm}). Hence, it has a well-defined stable and unstable manifold $\cM^s(\bfs_n^{n+k})$ and $\cM^u(\bfs_n^{n+k})$ respectively in $\Omega_n \cup \Gamma_n$. Define the local stable/unstable manifold of $\bfs_n^{n+k}$ as
$$
\cM^{s/u}_{\loc}(\bfs_n^{n+k}) := \Phi_n^{n+k}(\cM^{s/u}_{\loc}(\bfs_{n+k})) \subset \cM^{s/u}(\bfs_n^{n+k}).
$$
Note that $\cM^s_{\loc}(\bfs_n^{n+k})$ is the connected component of $\cM^s(\bfs_n^{n+k}) \cap V_n$ that contains $\bfs_n^{n+k}$.

\begin{prop}\label{higher period stable}
The local stable manifold $\cM^s_{\loc}(\bfs_n^{n+k})$ of $\bfs_n^{n+k}$ for $B_n$ is the graph of the map
$$
y \mapsto \psi_n^{n+k}(y)
\hspace{5mm} \text{for} \hspace{5mm}
y \in V_n,
$$
where $\psi_n^{n+k}$ is analytic, and $\|(\psi_n^{n+k})'\| = O(a^{q_{2n}})$.
\end{prop}

\begin{proof}
By \thmref{universalitya}, we have
$$
a_n(x,y) = \eta_n(x) - a^{q_{2n}} \alpha(x) y(1+O(\rho^n)).
$$
We compute
$$
(a_n)_y^{-1}(x) = \eta_n^{-1}(x) + a^{q_{2n}} \check \alpha(x) y(1+O(\rho^n)),
$$
where
$$
\check{\alpha}(x) := \frac{\alpha(\eta_*^{-1}(x))}{\eta_*'(\eta_*^{-1}(x))}.
$$
Hence,
$$
D_{(x,y)}H_{n+1} = \begin{bmatrix}
((a_n)_y^{-1})'(x) & a^{q_{2n}} \check \alpha(x) (1+O(\rho^n))\\
0 & 1
\end{bmatrix}.
$$
For $\vec v = (s, 1)$ with $s = O(a^{q_{2(n+1)}})$, we have
$$
D_{(x,y)}\Phi_{n+1}(\vec v) = D_{\Lambda_{n+1}((x,y))}H_{n+1} (\lambda_{n+1} \vec v) = \lambda_{n+1}(t, 1),
$$
where
$$
t = a^{q_{2n}} \check \alpha(\lambda_{n+1}x)(1+O(\rho^n)).
$$
Therefore, for $\vec v' = (s', 1)$ with $s' = O(a^{q_{2(n+k)}})$, we have
$$
D_{(x,y)}\Phi_n^{n+k}(\vec v') = D_{\Phi_{n+1}^{n+k}((x,y))}\Phi_{n+1}\ldots D_{(x,y)}\Phi_{n+k}(\vec v') = \lambda_{n+1} \ldots \lambda_{n+k}(t', 1),
$$
where
$$
t' = a^{q_{2n}} \check \alpha(\lambda_{n+1}x')(1+O(\rho^n))
$$
with $x'$ equal to the first coordinate of $\Phi_{n+1}^{n+k}((x,y))$. The result follows.
\end{proof}

By \propref{2dcap}, the microscope maps $\Phi_n^{n+k}$ converge locally uniformly to the constant map $(x,y) \mapsto (\kappa_n, 0)$ as $k \to \infty$, where $(\kappa_n, 0) \in \Omega_n$ is the $n$th fold. Since the family $\{\Phi_n^{n+k}\}_{n, k \in \bbN}$ depends holomorphically on the Jacobian $a \in \bbD_{\bar\epsilon} \setminus\{0\}$, we conclude:

\begin{prop}
The $n$th fold $(\kappa_n, 0)$ moves holomorphically under change in the Jacobian $a \in \bbD_{\bar\epsilon} \setminus\{0\}$.
\end{prop}

\begin{prop}\label{higher period location}
We have
$$
\psi_n^{n+k}(0) = \kappa_n + \lambda_*^{2k}u_*(x_*-1)(1+O(\rho^n)+O(\rho^k)),
$$
where $0 < \rho <1$ is a uniform constant; $\psi_n^{n+k}$ is the analytic function in \propref{higher period stable};  $u_*$ is given in \propref{limit micro}; and $x_*$ is the unique repelling fixed point of $\xi_*$ in $W$.
\end{prop}

\begin{proof}
By \propref{higher period stable}, we have
$$
\psi_n^{n+k}(0) - x_n^{n+k} = O(a^{q_{2n}}).
$$
Recall
$$
x_n^{n+k} - \kappa_n = \phi_n^{n+k}(x_{n+k}) - \kappa_n = \check \phi_n^{n+k}(x_{n+k} -\kappa_{n+k}),
$$
where $\check \phi_n^{n+k}$ is given in \eqref{eq:check phi x}. By \propref{limit phi x}, \ref{limit micro} and \ref{fixed1d}, we have
$$
(d_n^{n+k})^{-1}\check \phi_n^{n+k}(x_{n+k} -\kappa_{n+k}) = u_*(x_*-1)(1+O(\rho^k)).
$$
Lastly, by \propref{limit phi x}, we have
$$
d_n^{n+k} = \lambda_*^{2k}(1+O(\rho^n)).
$$
The result follows.
\end{proof}

\section{Vertical Tangency}

Consider the $n$th renormalization $$\bfR^n(F_a) = \Sigma_n = (A_n, B_n)$$ of the semi-Siegel H\'enon map $F_a$. Write the map $B_n : \Gamma_n = W \times V_n \to B_n(\Gamma_n)$ as
$$
B_n(x,y) = (b_n(x,y), x).
$$
Define $\xi_n : W \to \xi_n(W)$ by
$$
\xi_n(x) := b_n(x,0).
$$
Recall that by \thmref{convergence}, $\xi_n$ converges exponentially fast to $\xi_*$, where $\zeta_* = (\eta_*, \xi_*)$ is the fixed point of the one-dimensional renormalization operator $\cR$. Furthermore, by \thmref{universalityb}, we have
$$
b_n(x,y) = \xi_n(x) - a^{q_{2n}} \, \beta(x) \, y \, (1+O(\rho^n)).
$$
Let $\cU \subset W$ be the domain chosen in \secref{sec:local manifolds}.

\begin{lem}\label{graph bending}
For any analytic function $\chi : \cU \to V_n$, let $G_\chi$ be the graph of the map
$$
x \mapsto \chi(x)
\hspace{5mm} \text{for} \hspace{5mm}
x \in \cU.
$$
Then $B_n(G_\chi)$ is the graph of the map
$$
y \mapsto B_n^\chi(y) := b_n(y, \chi(y))
\hspace{5mm} \text{for} \hspace{5mm}
y \in \cU.
$$
Moreover, $B_n^\chi$ is $O(a^{q_{2n}})$-close to $\xi_n$. Consequently, there exists a unique point
$$
\bfv = (v,w) := (B_n^\chi(w), w)
$$
in $B_n(G_\chi)$ with a vertical tangent, which is $O(a^{q_{2n}})$-close to $(1,0)$. Lastly, there exists a uniform constant $0<\rho<1$ such that
\begin{equation}\label{eq:graph bending}
B_n^\chi(y) - B_n^\chi(w) = \frac{1}{2}\xi_*''(0)(y-w)^2(1+O(\rho^n)) + O((y - w)^3)
\hspace{5mm} \text{for} \hspace{5mm}
y \in D_\delta(0),
\end{equation}
where $D_\delta(0)$ is a disc of some uniform radius $\delta>0$ centered at $0$.
\end{lem}

\begin{proof}
The first claim follows from the fact that $B_n$ maps vertical lines to horizontal lines. By \thmref{universalityb}, $B_n^\chi$ is $O(a^{q_{2n}})$-close to $\xi_n$. Equation \eqref{eq:graph bending} follows from \thmref{convergence}.
\end{proof}

\begin{lem}\label{graph distance}
Let $\chi : \cU \to V_n$ and $\tilde\chi : \cU \to V_n$ be analytic functions. Then there exists a uniform constant $0< \rho < 1$ such that
$$
B_n^{\chi}(y) - B_n^{\tilde\chi}(y) = -a^{q_{2n}} \beta(y)(\chi(y) - \tilde\chi(y))(1+O(\rho^n))
\hspace{5mm} \text{for} \hspace{5mm}
y \in \cU.
$$
\end{lem}

\begin{proof}
The result is an immediate consequence of \thmref{universalityb}.
\end{proof}

Recall that the local unstable manifold $\cM^u_{\loc}(\bfs_n) \subset \Gamma_n$ of the saddle fixed point $\bfs_n$ for $B_n$ is the graph of the map
$$
x \mapsto \chi_n(x)
\hspace{5mm} \text{for} \hspace{5mm}
x \in \cU,
$$
where $\chi_n$ is given in \propref{fixed unstable}. By \lemref{graph bending}, the set $B_n(\cM^u(\bfs_n))$ is the graph of the map
$$
y \mapsto B_n^{\chi_n}(y) = b_n(y, \chi_n(y))
\hspace{5mm} \text{for} \hspace{5mm}
y \in \cU.
$$
The unique point $\bfv_n = (v_n, w_n)$ in $B_n(\cM^u_{\loc}(\bfs_n))$ that has a vertical tangent is called the {\it vertical tangency} for $B_n$. See \figref{fig:verticaltangent}.

The local unstable manifold $\cM^u_{\loc}(\bfs_{n+k}) \subset \Gamma_{n+k}$ of the saddle fixed point $\bfs_{n+k}$ for $B_{n+k}$ is mapped by the renormalization microscope map $\Phi_n^{n+k}$ to the local unstable manifold $\cM^u_{\loc}(\bfs_n^{n+k}) \subset \Gamma_n$ of a saddle fixed point $\bfs_n^{n+k}$ for the iterate $pB_n^{n+k}$ of $\Sigma_n$. By \eqref{eq:prerenorm}, we have
$$
\Phi_n^{n+k} \circ B_{n+k}(\cM^u_{\loc}(\bfs_{n+k})) = pB_n^{n+k}(\cM^u_{\loc}(\bfs_n^{n+k})).
$$
Consider the vertical tangency $\bfv_{n+k} \in B_{n+k}(W^u_{\loc}(\bfs_{n+k}))$ for $B_{n+k}$. Denote
$$
\bfv_n^{n+k} = (v_n^{n+k}, w_n^{n+k}) := \Phi_n^{n+k}(\bfv_{n+k}) \in pB_n^{n+k}(\cM^u_{\loc}(\bfs_n^{n+k})).
$$

\begin{rem}
The point $\bfv_n^{n+k}$ does not have a vertical tangent (even though it is the image of a point that does). Indeed, by \thmref{cap derivative}, we see that the tangent of $\bfv_n^{n+k}$ is tilted by an angle of the order of $|a|^{q_{2n}}$. See \figref{fig:verticaltangent}.
\end{rem}

\begin{prop}\label{deeper unstable}
The set $A_n(\cM^u_{\loc}(\bfs_n^{n+1}))$ is the graph of the map
$$
x \mapsto \sigma_n(x)
\hspace{5mm} \text{for} \hspace{5mm}
x \in \lambda_*\cU,
$$
where $\sigma_n$ is analytic and $O(a^{q_{2(n-1)}})$-close to $(\xi_n^{\rot})^{-1}$. Consequently, the set
$$
pB_n^{n+1}(\cM^u_{\loc}(\bfs_n^{n+1})) = B_n \circ A_n(\cM^u_{\loc}(\bfs_n^{n+1}))
$$
is the graph of the map
$$
y \mapsto B_n^{\sigma_n}(y)
\hspace{5mm} \text{for} \hspace{5mm}
x \in \lambda_* \cU.
$$
Moreover, there exists a unique point
$$
\bar\bfv_n^{n+1} = (\bar v_n^{n+1}, \bar w_n^{n+1}) := (B_n^{\sigma_n}(\bar w_n^{n+1}), \bar w_n^{n+1})
$$
in $pB_n^{n+1}(\cM^u_{\loc}(\bfs_n^{n+1}))$ with a vertical tangent, which is $O(a^{q_{2n}})$-close to $(1,0)$.
\end{prop}

\begin{proof}
The first claim is an immediate consequence of \thmref{universalityaphi}. The rest follows from \lemref{graph bending}.
\end{proof}

\begin{prop}\label{vertical tangency distance}
Consider the points
$$
\bfv_n = (v_n, w_n)
\hspace{5mm} \text{and} \hspace{5mm}
\bfv_n^{n+1} = (v_n^{n+1}, w_n^{n+1}).
$$
There exist uniform constants $0 < \rho <1$ and $\Delta_w, \Delta_v \in \bbC \setminus \{0\}$ such that
$$
w_n - w_n^{n+1} = a^{q_{2n}}\Delta_w(1+O(\rho^n))
$$
and
$$
v_n - v_n^{n+1} = a^{q_{2n}}\Delta_v(1+O(\rho^n)).
$$
\end{prop}

\begin{proof}
Letting $\chi = \chi_n$ and $\tilde{\chi} = \sigma_n$ in \lemref{graph distance} and differentiating, we obtain
\begin{equation}\label{eq: vt distance 1}
(B_n^{\chi_n})'(y) - (B_n^{\sigma_n})'(y) = a^{q_{2n}}\big(\beta(y)(\chi_n(y)-\sigma_n(y)) + \beta'(y)(\chi_n'(y) - \sigma_n'(y))\big).
\end{equation}
Denote
$$
\tilde\Delta_w = \beta(0)((\xi_*^{\exp})^{-1}(0) - (\xi_*^{\rot})^{-1}(0)) + \beta'(0)((\xi_*^{\exp})^{-1})'(0) - ((\xi_*^{\rot})^{-1})'(0)).
$$
Plugging in $y = w_n = O(a^{q_{2n}})$ into \eqref{eq: vt distance 1} and using \remref{convergence of inverses}, we obtain
\begin{equation}\label{eq:vt distance 2}
(B_n^{\sigma_n})'(w_n) = a^{q_{2n}}\tilde\Delta_w(1+O(\rho^n)).
\end{equation}

Consider the vertical tangency $\bfv_{n+1}$ of $B_{n+1}$. By \thmref{cap derivative}, the derivative $D_{\bfv_{n+1}}\Phi_{n+1}$ maps the vertical tangent vector $\vec v = (0,1)$ of $B_{n+1}(\cM^u_{\loc}(\bfs_{n+1}))$ at $\bfv_{n+1}$ to the tangent vector
$$
D_{\bfv_{n+1}}\Phi_{n+1}(\vec v) = \lambda_{n+1}\begin{bmatrix}
\lambda_*\alpha(1)a^{q_{2n}}(1+O(\rho^n))\\
1
\end{bmatrix}
$$
of $pB_n^{n+1}(\cM^u_{\loc}(\bfs_n^{n+1}))$ at $\bfv_n^{n+1}$. It follows that
\begin{equation}\label{eq:vt distance 3}
(B_n^{\sigma_n})'(w_n^{n+1}) = \lambda_*\alpha(1)a^{q_{2n}}(1+O(\rho^n)).
\end{equation}

Letting $\chi = \sigma_n$ in \lemref{graph bending} and differentiating \eqref{eq:graph bending} with $w = \bar w_n^{n+1}$, we obtain
\begin{equation}\label{eq:vt distance 4}
(B_n^{\sigma_n})'(y) = \xi_*''(0)(y - \bar w_n^{n+1})(1+O(\rho^n)) + O((y-\bar w_n^{n+1})^2).
\end{equation}
Combining \eqref{eq:vt distance 2}, \eqref{eq:vt distance 3} and \eqref{eq:vt distance 4}, we have
\begin{equation}\label{eq:vt distance 5}
w_n - \bar w_n^{n+1} = O(a^{q_{2n}})
\hspace{5mm} \text{and} \hspace{5mm}
w_n^{n+1} - \bar w_n^{n+1} = O(a^{q_{2n}}),
\end{equation}
and
$$
w_n - w_n^{n+1} = a^{q_{2n}}\Delta_w(1+O(\rho^n)),
$$
where
\begin{equation}\label{eq:vt distance 6}
\Delta_w := \frac{\tilde \Delta_w - \lambda_*\alpha(1)}{\xi_*''(0)}.
\end{equation}

By \eqref{eq:vt distance 5} and \eqref{eq:graph bending}, we have
$$
B_n^{\sigma_n}(w_n) - \bar v_n^{n+1} = O(a^{2q_{2n}})
\hspace{5mm} \text{and} \hspace{5mm}
v_n^{n+1} - \bar v_n^{n+1} = O(a^{2q_{2n}}).
$$
Therefore,
\begin{align*}
v_n - v_n^{n+1} &= v_n - \bar v_n^{n+1} + O(a^{2q_{2n}})\\
&= B_n^{\chi_n}(w_n) - B_n^{\sigma_n}(w_n) + O(a^{2q_{2n}})\\
&= -a^{q_{2n}} \beta(w_n)(\chi_n(w_n) - \sigma_n(w_n))(1+O(\rho^n))\\
&=a^{q_{2n}}\Delta_v(1+O(\rho^n)),
\end{align*}
where in the third equality we have used \lemref{graph distance}, and in the last equality we let
$$
\Delta_v: = -\beta(0)\big((\xi_*^{\exp})^{-1}(0) - (\xi_*^{\rot})^{-1}(0)\big).
$$

\end{proof}

\begin{cor}\label{vertical tangency location}
Consider the vertical tangency $\bfv_n = (v_n, w_n) \in B_n(\cM^u_{\loc}(\bfs_n))$ for $B_n$. Let $\Delta_w, \Delta_v \in \bbC \setminus \{0\}$ be the constants given in \propref{vertical tangency distance}. Then there exists a uniform constant $0 < \rho <1$ such that
$$
w_n = a^{q_{2n}}\Delta_w(1+O(\rho^n))
$$
and
$$
v_n - \kappa_n = a^{q_{2n}}\Delta_v(1+O(\rho^n)).
$$
\end{cor}

\begin{proof}
Applying \propref{vertical tangency distance} to scale $n+k$, we obtain:
$$
\bfv_{n+k} - \bfv_{n+k}^{n+k+1} = O(a^{q_{2(n+k)}}).
$$
Hence,
$$
\bfv_n^{n+k} - \bfv_n^{n+k+1} = \Phi_n^{n+k}(\bfv_{n+k}) - \Phi_n^{n+k}(\bfv_{n+k}^{n+k+1}) = O(\lambda_*^k a^{q_{2(n+k)}}).
$$
Recall that $\bfv_n^{n+k} \to (\kappa_n, 0)$ as $k \to \infty$. Observe
\begin{align*}
\bfv_n - (\kappa_n, 0) &= \bfv_n - \bfv_n^{n+1} + \sum_{k=1}^\infty \big(\bfv_n^{n+k} - \bfv_n^{n+k+1}\big)\\
&= \bfv_n - \bfv_n^{n+1} + O(a^{q_{2(n+1)}}).
\end{align*}
The result now follows from \propref{vertical tangency distance}.
\end{proof}

Let $\hat{\cC}_{n+1}$ be the connected component of $B_n(\cM^u_{\loc}(\bfs_n)) \cap (\Omega_n^{n+1} \cup \Gamma_n^{n+1})$. Note that $\bfv_n \in \hat{\cC}_{n+1}$. Denote
\begin{equation}\label{eq:previous unstable}
\cC_{n+1}:= \Phi_{n+1}^{-1}(\hat{\cC}_{n+1}) \subset \Omega_{n+1} \cup  \Gamma_{n+1}
\end{equation}
and
$$
\bfv_{n+1}^n = (v_{n+1}^n, w_{n+1}^n) := \Phi_{n+1}^{-1}(\bfv_n) \in \cC_{n+1}.
$$

\begin{cor}\label{previous unstable}
The set $\cC_{n+1}$ is the graph of the map
$$
y \mapsto \tau_{n+1}(y)
\hspace{5mm} \text{for} \hspace{5mm}
y \in \lambda_*^{-1}\cU,
$$
where $\tau_{n+1}$ is analytic, and $O(a^{q_{2n}})$-close to $\xi_{n+1}$. Consequently, there exists a unique point
$$
\bar{\bfv}_{n+1}^n = (\bar v_{n+1}^n, \bar w_{n+1}^n)
$$
in $\cC_{n+1}$ with a vertical tangency. Lastly, we have
$$
\bar w_{n+1}^n = a^{q_{2n}}\bar{\Delta}_w(1+O(\rho^n))
$$
and
$$
\bar v_{n+1}^n - \kappa_{n+1} = a^{q_{2n}}\bar{\Delta}_v(1+O(\rho^n)),
$$
where $0 < \rho<1$ and $\bar{\Delta}_w, \bar{\Delta}_v \in \bbC \setminus \{0\}$ are uniform constants.
\end{cor}

\begin{proof}
Let $\bfx_0 = (x_0, y_0) \in B_n(\cM^u_{\loc}(\bfs_n))$. Denote
$$
\bfx_1 = (x_1, y_1) := \Phi_{n+1}^{-1}(\bfx_0) = \Lambda_{n+1}^{-1} \circ H_{n+1}^{-1}(\bfx_0).
$$
By \thmref{universalityb} and \ref{universalitya}, we have
$$
\bfx_0 = (x_0, y_0) = \begin{bmatrix}
\xi_n(y_0) + O(a^{q_{2n}})\\
y_0
\end{bmatrix}
\hspace{5mm} , \hspace{5mm}
H_{n+1}^{-1}(\bfx_0) = \begin{bmatrix}
\eta_n \circ \xi_n(y_0) + O(a^{q_{2n}})\\
y_0
\end{bmatrix}
$$
and
\begin{align}
\bfx_1 = (x_1, y_1) &= \begin{bmatrix}
\lambda_{n+1}^{-1}\eta_n \circ \xi_n(y_0) + O(a^{q_{2n}})\\
\lambda_{n+1}^{-1}y_0
\end{bmatrix}\nonumber\\
&= \begin{bmatrix}
\lambda_{n+1}^{-1}\eta_n \circ \xi_n(\lambda_{n+1}y_1) + O(a^{q_{2n}})\\
y_1
\end{bmatrix}.\label{eq:previous unstable 1}
\end{align}

Recall
$$
B_{n+1}(x,0) = (\xi_{n+1}(x), x) = \Phi_{n+1}^{-1} \circ B_n \circ A_n \circ \Phi_{n+1}(x,0).
$$
A straightforward computation using \thmref{universalityaphi}, \thmref{universalityb} and the definition of $\Phi_{n+1}$ yields
$$
\xi_{n+1}(x) =  \lambda_{n+1}^{-1}\eta_n \circ \xi_n(\lambda_{n+1}x) + O(a^{q_{2n}}).
$$
Thus, by \eqref{eq:previous unstable 1}, we see that $\cC_{n+1}$ is the graph of the map
$$
y \mapsto \tau_{n+1}(y)
\hspace{5mm} \text{for} \hspace{5mm}
y \in \lambda_*^{-1}\cU,
$$
where $\tau_{n+1}$ is $O(a^{q_{2n}})$-close to $\xi_{n+1}$.

Let
$$
\hat{\bfv}_n^{n+1} = (\hat v_n^{n+1}, \hat w_n^{n+1}) = (B_n^{\chi_n}(\hat w_n^{n+1}), \hat w_n^{n+1}):=\Phi_{n+1}(\bar \bfv_{n+1}^n).
$$
By \thmref{cap derivative}, the derivative $D_{\bar\bfv_{n+1}^n}\Phi_{n+1}$ maps the vertical tangent vector $\vec v = (0,1)$ of $\cC_{n+1}$ at $\bar\bfv_{n+1}^n$ to the tangent vector
$$
D_{\bfv_{n+1}}\Phi_{n+1}(\vec v) = \lambda_{n+1}\begin{bmatrix}
\lambda_*\alpha(1)a^{q_{2n}}(1+O(\rho^n))\\
1
\end{bmatrix}
$$
of $B_n(\cM^u_{\loc}(\bfs_n))$ at $\hat\bfv_n^{n+1}$. It follows that
$$
(B_n^{\chi_n})'(\hat w_n^{n+1}) = \lambda_*\alpha(1)a^{q_{2n}}(1+O(\rho^n)).
$$
Letting $\chi = \chi_n$ in \lemref{graph bending} and differentiating \eqref{eq:graph bending} with $w = w_n$, we obtain
\begin{equation}\label{eq:previous unstable 2}
\hat w_n^{n+1} - w_n = \frac{\lambda_*\alpha(1)}{\xi_*''(0)}a^{q_{2n}}(1+O(\rho^n)).
\end{equation}
\corref{vertical tangency location} and \eqref{eq:vt distance 6} now implies
\begin{align}
\hat w_n^{n+1} &= \left(\frac{\lambda_*\alpha(1)}{\xi_*''(0)} + \Delta_w\right)a^{q_{2n}}(1+O(\rho^n))\nonumber\\
&= \frac{\tilde\Delta_w}{\xi_*''(0)}a^{q_{2n}}(1+O(\rho^n)).\label{eq:previous unstable 3}
\end{align}

By \eqref{eq:graph bending} and \eqref{eq:previous unstable 2}, we have
$$
\hat v_n^{n+1} - v_n = O(a^{2q_{2n}}).
$$
Thus, it follows from \corref{vertical tangency location} that
\begin{equation}\label{eq:previous unstable 4}
\hat v_n^{n+1} - \kappa_n = a^{q_{2n}}\Delta_v(1+O(\rho^n)).
\end{equation}

Finally, by applying \thmref{cap derivative} to \eqref{eq:previous unstable 3} and \eqref{eq:previous unstable 4}, we obtain
$$
\bar w_{n+1}^n = a^{q_{2n}}\bar \Delta_w(1+O(\rho^n))
\hspace{5mm} \text{and} \hspace{5mm}
\bar v_{n+1}^n - \kappa_{n+1} = a^{q_{2n}}\bar \Delta_v(1+O(\rho^n))
$$
where
$$
\bar \Delta_w:= \frac{\tilde\Delta_w}{\lambda_*\xi_*''(0)}
\hspace{5mm} \text{and} \hspace{5mm}
\bar \Delta_v := \frac{\Delta_v}{\lambda_*^2}.
$$
\end{proof}

\begin{figure}[h]
\centering
\includegraphics[scale=0.24]{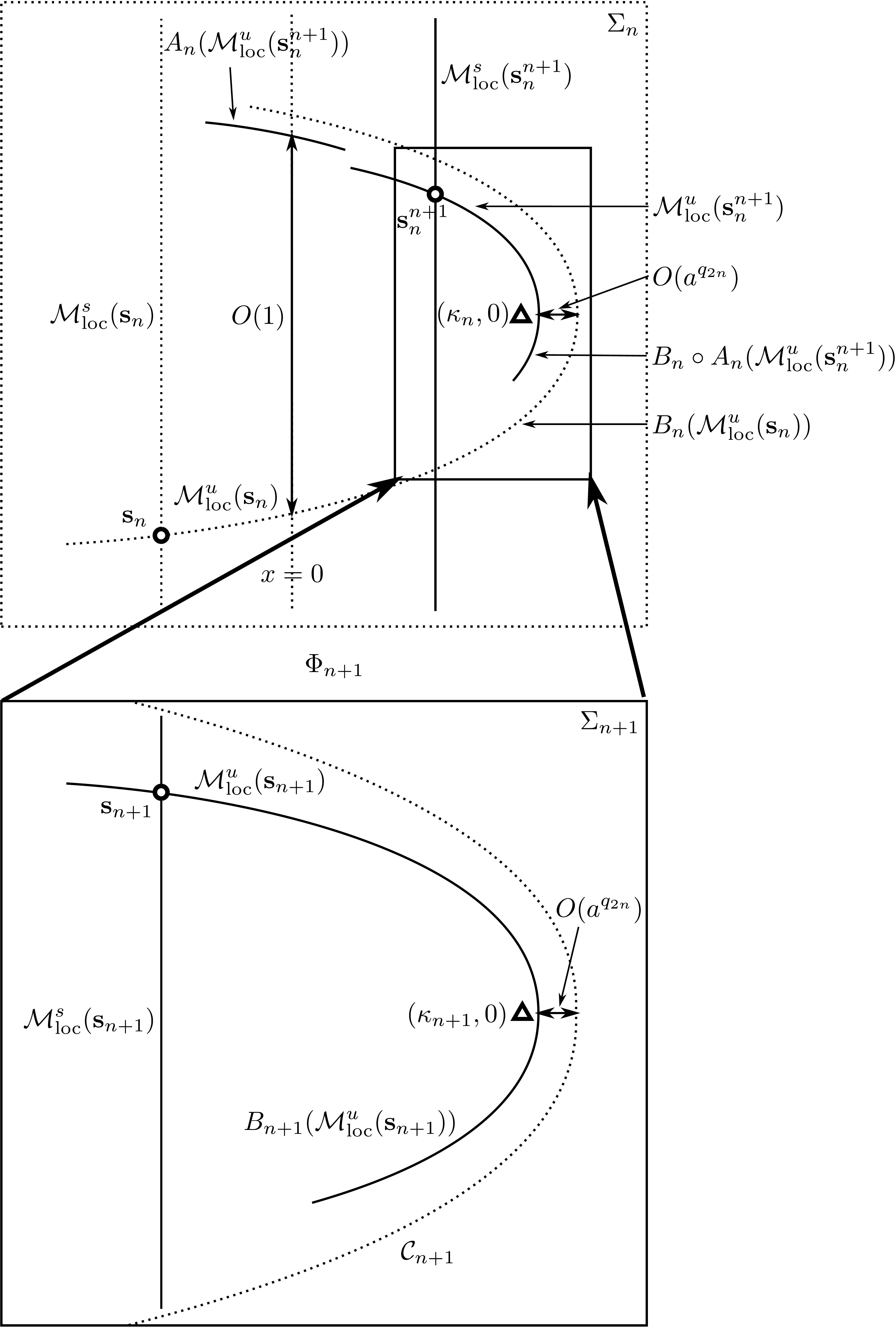}
\caption{For the pair $\Sigma_n$, we have the unstable manifolds $\cM^u_{\loc}(\bfs_n)$ and $\cM^u_{\loc}(\bfs_n^{n+1})$, and their images $B_n(\cM^u_{\loc}(\bfs_n))$, $A_n(\cM^u_{\loc}(\bfs_n^{n+1}))$ and $B_n \circ A_n(\cM^u_{\loc}(\bfs_n^{n+1}))$. Under $B_n$, the $O(1)$-vertical distance between $\cM^u_{\loc}(\bfs_n)$ and $A_n(\cM^u_{\loc}(\bfs_n^{n+1}))$ near $x=0$ shrinks to a $O(a^{q_{2n}})$-horizontal distance between $B_n(\cM^u_{\loc}(\bfs_n))$ and $B_n \circ A_n(\cM^u_{\loc}(\bfs_n^{n+1}))$ near $(\kappa_n, 0)$. For the pair $\Sigma_{n+1}$, we have the unstable manifolds $B_{n+1}(\cM^u_{\loc}(\bfs_{n+1}))$ and $\cC_{n+1}$, which map into $B_n \circ A_n(\cM^u_{\loc}(\bfs_n^{n+1}))$ and $B_n(\cM^u_{\loc}(\bfs_n))$ respectively under $\Phi_{n+1}$.}
\label{fig:verticaltangent}
\end{figure}

\section{Heteroclinic Tangency}

Consider the $n$th renormalization $\bfR^n(F_a) = \Sigma_n = (A_n, B_n)$ of the semi-Siegel H\'enon map $F_a$. The sequence of saddle points $\bfs_n^{n+k}$, which are fixed by the iterate $pB_n^{n+k}$ of $\Sigma_n$, converges to the $n$th fold $(\kappa_n, 0)$. Recall that by \propref{higher period stable}, the local stable manifold $\cM^s_{\loc}(\bfs_n^{n+k})$ is the graph of the map
$$
y \mapsto \psi_n^{n+k}(y)
\hspace{5mm} \text{for} \hspace{5mm}
y \in V_n,
$$
where $\|(\psi_n^{n+k})'\| = O(a^{q_{2n}})$. See \figref{fig:heterotangent}.

Next, consider the set $\cC_n$ defined in \eqref{eq:previous unstable}. By \corref{previous unstable}, $\cC_n$ is the graph of the map
$$
y \mapsto \tau_n(y)
\hspace{5mm} \text{for} \hspace{5mm}
y \in \lambda_*^{-1}\cU,
$$
where $\tau_n$ is $O(a^{q_{2(n-1)}})$-close to $\xi_n$. Recall that there is a unique point
$$
\bar \bfv_n^{n-1} = (\bar v_n^{n-1}, \bar w_n^{n-1})
$$
in $\cC_n$ that has a vertical tangency. See \figref{fig:heterotangent}.

\begin{figure}[h]
\centering
\includegraphics[scale=0.3]{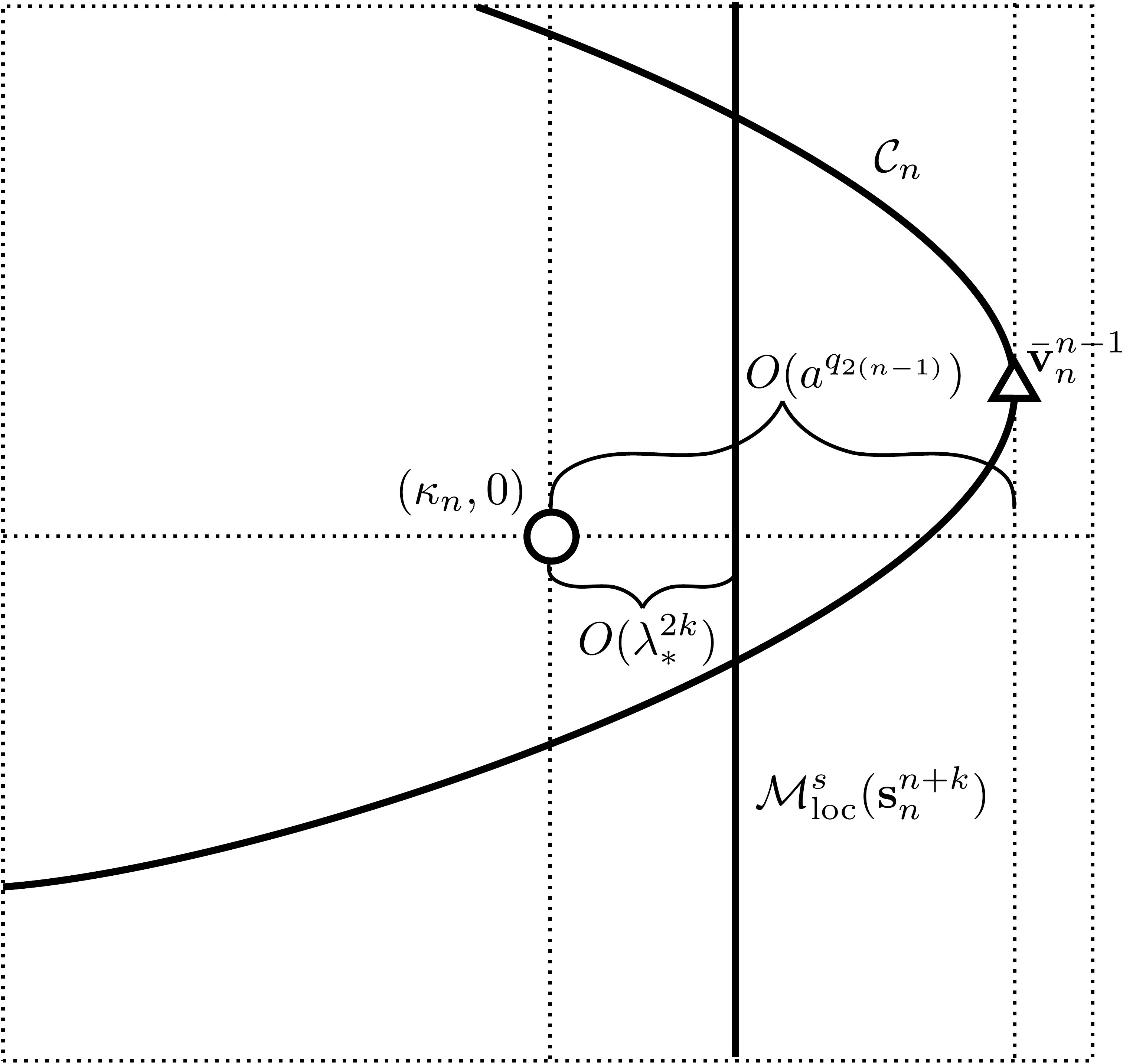}
\caption{The stable manifold $\cM^s_{\loc}(\bfs_n^{n+k})$, which is nearly vertical, and the unstable manifold $\cC_n$, which is quadratic. The horizontal position of $\cM^s_{\loc}(\bfs_n^{n+k})$ is given by \propref{higher period location}, and the position of the vertical tangency $\bar \bfv_n^{n-1}$ in $\cC_n$ is given by \corref{previous unstable}.}
\label{fig:heterotangent}
\end{figure}

Observe that a point
$$
\bfq_{n,k} = (q_{n,k}, r_{n,k})
$$
is a heteroclinic intersection between $\cM^s_{\loc}(\bfs_n^{n+k})$ and $\cC_n$ if and only if
$$
q_{n,k} = \psi_n^{n+k}(r_{n,k}) = \tau_n(r_{n,k}).
$$
Moreover, $\bfq_{n,k}$ is a heteroclinic tangency if and only if additionally, we have
$$
(\psi_n^{n+k})'(r_{n,k}) = \tau_n'(r_{n,k}).
$$

\begin{proof}[Proof of the Main Theorem]
Observe
$$
\|\tau_n''\| = \|\xi_n''\| + O(a^{q_{2(n-1)}}) > C
$$
for some uniform constant $C>0$, while
$$
\|(\psi_n^{n+k})'\| = O(a^{q_{2n}}).
$$
Since the vertical tangency $\bar w_n^{n-1}$ is the unique point such that
$$
\tau_n'(\bar w_n^{n-1}) = 0,
$$
there exists a unique point $r_{n,k}$ which is $O(a^{q_{2n}})$-close to $\bar w_n^{n-1}$ such that
$$
(\tau_n)'(r_{n,k}) - (\psi_n^{n+k})'(r_{n,k}) = 0.
$$
Moreover, by \eqref{eq:graph bending} with $B_n^\chi = \tau_n$ and $w = \bar w_n^{n-1}$, we see that
\begin{equation}\label{eq:main thm 1}
\tau_n(r_{n,k}) - \bar v_n^{n-1} = O(a^{2q_{2n}}).
\end{equation}

By \corref{previous unstable} and \eqref{eq:main thm 1}, we have
$$
\tau_n(r_{n,k}) = \kappa_n + a^{q_{2(n-1)}}\bar{\Delta}_v(1+O(\rho^n)).
$$
Moreover, by \propref{higher period stable} and \ref{higher period location}, we have
$$
\psi_n^{n+k}(r_{n,k}) = \kappa_n + \lambda_*^{2k}u_*(x_*-1)(1+O(\rho^n)+O(\rho^k)).
$$
Therefore,
\begin{equation}\label{eq:main thm 2}
\tau_n(r_{n,k}) = \psi_n^{n+k}(r_{n,k}) \iff a^{q_{2(n-1)}} = \lambda_*^{2k}\frac{u_*(x_*-1)}{\bar{\Delta}_v}(1+O(\rho^n)+O(\rho^n)).
\end{equation}
Observe that the set of solutions to \eqref{eq:main thm 2} as $n$ and $k$ run through $\bbN$ forms a discrete dense subset $X$ of  $\bbD_{\bar \epsilon} \setminus \{0\}$. This means that in any open set $U \subset \bbD_{\bar \epsilon} \setminus \{0\}$, there exists a parameter $a_1 \in U \cap X$ such that the semi-Siegel H\'enon map $F_{a_1}$ has a heteroclinic tangency $\bfq_{n,k}$ that does not persist in any neighborhood of $a_1$. By \thmref{weak stability} v), it follows that the semi-Siegel family $(F_a)_{a \in \bbD_{\bar\epsilon}\setminus\{0\}}$ is not weakly $J^*$-stable at every parameter.
\end{proof}

\end{document}